\documentclass[reqno]{amsart}
\usepackage{amssymb}
\usepackage{amsfonts}

\newtheorem{theorem}{Theorem}
\theoremstyle{plain}

\newtheorem{corollary}{Corollary}

\newtheorem{definition}{Definition}

\newtheorem{lemma}{Lemma}

\newtheorem{problem}{Problem}
\newtheorem{proposition}{Proposition}
\newtheorem{remark}{Remark}

\numberwithin{equation}{section}
\numberwithin{theorem}{section}
\numberwithin{algorithm}{section}
\numberwithin{axiom}{section}
\numberwithin{case}{section}
\numberwithin{claim}{section}
\numberwithin{conclusion}{section}
\numberwithin{condition}{section}
\numberwithin{conjecture}{section}
\numberwithin{corollary}{section}
\numberwithin{criterion}{section}
\numberwithin{definition}{section}
\numberwithin{example}{section}
\numberwithin{exercise}{section}
\numberwithin{lemma}{section}
\numberwithin{notation}{section}
\numberwithin{problem}{section}
\numberwithin{proposition}{section}
\numberwithin{remark}{section}
\numberwithin{solution}{section}

\input{tcilatex}

\begin{document}
\title{Lectures on the fourth order $Q$ curvature equation}
\author{Fengbo Hang}
\address{Courant Institute, New York University, 251 Mercer Street, New York
NY 10012}
\email{fengbo@cims.nyu.edu}
\author{Paul C. Yang}
\address{Department of Mathematics, Princeton University, Fine Hall,
Washington Road, Princeton NJ 08544}
\email{yang@math.princeton.edu}
\date{}
\dedicatory{}

\begin{abstract}
We discuss some open problems and recent progress related to the $4$th order
Paneitz operator and $Q$ curvature in dimensions other than $4$.
\end{abstract}

\maketitle

\section{Introduction\label{sec1}}

In conformal geometry, a major tool is a family of conformal covariant
operators and their associated curvature invariants. In dimension $n>2$, the
conformal Laplacian operator%
\begin{equation}
L=-\frac{4\left( n-1\right) }{n-2}\Delta +R  \label{eq1.1}
\end{equation}%
enjoys the following covariance property,%
\begin{equation}
{L}_{\rho ^{\frac{4}{n-2}}g}\varphi =\rho ^{-\frac{n+2}{n-2}}L_{g}(\rho
\varphi )  \label{eq1.2}
\end{equation}%
for any smooth positive function $\rho $ (see \cite{LP}). Here $R$ denotes
the scalar curvature. The associated transformation law of scalar curvature
follows,%
\begin{equation}
{R}_{\rho ^{\frac{4}{n-2}}g}={L}_{\rho ^{\frac{4}{n-2}}g}1=\rho ^{-\frac{n+2%
}{n-2}}L_{g}\rho .  \label{eq1.3}
\end{equation}%
A fundamental result is the solution of the Yamabe problem \cite{Au2, S, T,
Y}, which is related to the sharp constant of the associated Sobolev
inequality. Since then, there is a large literature on the analysis and
geometry of this equation. In order to gain additional information on the
Ricci tensor, the $4$th order $Q$ curvature equation comes into play.

Let $\left( M,g\right) $ be a smooth Riemannian manifold with dimension $%
n\geq 3$, the $Q$ curvature is given by (\cite{B,P})%
\begin{eqnarray}
Q &=&-\frac{1}{2\left( n-1\right) }\Delta R-\frac{2}{\left( n-2\right) ^{2}}%
\left\vert Rc\right\vert ^{2}+\frac{n^{3}-4n^{2}+16n-16}{8\left( n-1\right)
^{2}\left( n-2\right) ^{2}}R^{2}  \label{eq1.4} \\
&=&-\Delta J-2\left\vert A\right\vert ^{2}+\frac{n}{2}J^{2}.  \notag
\end{eqnarray}%
Here $Rc$ is the Ricci tensor and%
\begin{equation}
J=\frac{R}{2\left( n-1\right) },\quad A=\frac{1}{n-2}\left( Rc-Jg\right) .
\label{eq1.5}
\end{equation}%
The Paneitz operator is defined as%
\begin{eqnarray}
&&P\varphi  \label{eq1.6} \\
&=&\Delta ^{2}\varphi +\frac{4}{n-2}\func{div}\left( Rc\left( \nabla \varphi
,e_{i}\right) e_{i}\right) -\frac{n^{2}-4n+8}{2\left( n-1\right) \left(
n-2\right) }\func{div}\left( R\nabla \varphi \right) +\frac{n-4}{2}Q\varphi 
\notag \\
&=&\Delta ^{2}\varphi +\func{div}\left( 4A\left( \nabla \varphi
,e_{i}\right) e_{i}-\left( n-2\right) J\nabla \varphi \right) +\frac{n-4}{2}%
Q\varphi .  \notag
\end{eqnarray}%
Here $e_{1},\cdots ,e_{n}$ is a local orthonormal frame with respect to $g$.
Note that the use of $J$ and $A$ (Schouten tensor) simplifies the formulas
of $Q$ curvature and Paneitz operator.

In dimension $n\neq 4$, the operator satisfies%
\begin{equation}
P_{\rho ^{\frac{4}{n-4}}g}\varphi =\rho ^{-\frac{n+4}{n-4}}P_{g}\left( \rho
\varphi \right)  \label{eq1.7}
\end{equation}%
for any positive smooth function $\rho $. This is similar to (\ref{eq1.2}).
As a consequence we have%
\begin{equation}
Q_{\rho ^{\frac{4}{n-4}}g}=\frac{2}{n-4}P_{\rho ^{\frac{4}{n-4}}g}1=\frac{2}{%
n-4}\rho ^{-\frac{n+4}{n-4}}P_{g}\rho .  \label{eq1.8}
\end{equation}

In dimension $4$, the Paneitz operator satisfies%
\begin{equation}
P_{e^{2w}g}\varphi =e^{-4w}P_{g}\varphi  \label{eq1.9}
\end{equation}%
and the $Q$ curvature transforms as%
\begin{equation}
Q_{e^{2w}g}=e^{-4w}\left( P_{g}w+Q_{g}\right) .  \label{eq1.10}
\end{equation}%
This should be compared to the conformal invariance of $-\Delta $ on surface
and the transformation law of Gaussian curvature under a conformal change of
metric.

The main theme of research is to find out the role of Paneitz operator and $%
Q $ curvature in understanding the geometry of a conformal class and the
topology of underlying manifold. For example we would like to know how the
spectral property of Paneitz operator affects the topology. Below we will
start with dimension $4$, when the $Q$ curvature equation and its
applications is relatively well understood. Then we will discuss recent
progress in dimension $n\geq 5$ about the Green's function of Paneitz
operator and the solution to finding constant $Q$ curvature in a fixed
conformal class. At last we will turn to the dimension $3$, where the $Q$
curvature equation is particularly intriguing and of very different nature
from the scalar curvature equation. Open problems will be pointed out along
the way.

\section{Dimension $4$\label{sec2}}

A basic fact that makes the $Q$ curvature interesting is its appearance in
the Chern-Gauss-Bonnet formula. For a closed $4$-manifold $\left( M,g\right) 
$ we have%
\begin{equation}
\int_{M}Qd\mu +\frac{1}{4}\int_{M}\left\vert W\right\vert ^{2}d\mu =8\pi
^{2}\chi \left( M\right) .  \label{eq2.1}
\end{equation}%
Here $W$ is the Weyl tensor. It follows from the pointwise conformal
invariance of $\left\vert W\right\vert ^{2}d\mu $ and (\ref{eq2.1}) that the 
$Q$ curvature integral is a global conformal invariant which we denote by $%
\kappa _{g}$ i.e.%
\begin{equation}
\kappa _{g}=\int_{M}Q_{g}d\mu _{g}  \label{eq2.2}
\end{equation}%
and $\kappa _{\widetilde{g}}=\kappa _{g}$ for any $\widetilde{g}\in \left[ g%
\right] $, the conformal class of $g$. A basic result about this invariant
is the following sharp upper bound:

\begin{theorem}[\protect\cite{G2}]
\label{thm2.1}Let $\left( M,g\right) $ be a smooth compact four manifold. If 
$L_{g}>0$, then $\kappa _{g}\leq 16\pi ^{2}$ with equality holds if and only
if $\left( M,g\right) $ is conformal diffeomorphic to the standard four
sphere.
\end{theorem}

Theorem \ref{thm2.1} follows from an identity found in \cite{HY4}. The
identity will have a crucial counterpart in other dimensions.

\begin{theorem}[\protect\cite{HY4}]
\label{thm2.2}Let $\left( M,g\right) $ be a $4$-dimensional smooth compact
Riemannian manifold with $L_{g}>0$. For $p\in M$, let $G_{L,p}$ be the
Green's function for $L_{g}$ with pole at $p$, then we have $\left\vert
Rc_{G_{L,p}^{2}g}\right\vert _{g}^{2}$ is bounded and%
\begin{equation}
P\left( \log G_{L,p}\right) =16\pi ^{2}\delta _{p}-\frac{1}{2}\left\vert
Rc_{G_{L,p}^{2}g}\right\vert _{g}^{2}-Q  \label{eq2.3}
\end{equation}%
in distribution sense.
\end{theorem}

Choosing $1$ as test function in (\ref{eq2.3}) we see%
\begin{equation*}
\int_{M}Qd\mu =16\pi ^{2}-\frac{1}{2}\int_{M}\left\vert
Rc_{G_{L,p}^{2}g}\right\vert _{g}^{2}d\mu \leq 16\pi ^{2}.
\end{equation*}%
If equality holds, then $Rc_{G_{L,p}^{2}g}=0$ and by the relative volume
comparison theorem we conclude $\left( M,g\right) $ must be conformal
equivalent to the standard $S^{4}$ (see \cite[section 5]{HY4}).

To study the $Q$ curvature equation, it is important that the Paneitz
operator be nonnegative with only constant functions in its kernel. A quite
general condition ensuring such kind of positivity is given by

\begin{theorem}[\protect\cite{G2}]
\label{thm2.3}Let $\left( M,g\right) $ be a smooth compact $4$-dimensional
Riemannian manifold with $L_{g}>0$ and $\kappa _{g}\geq 0$, then the Paneitz
operator $P\geq 0$ and the kernel of $P$ consists of constant functions.
\end{theorem}

As an application of Theorem \ref{thm2.1} and \ref{thm2.3}, we have a
general existence result for a conformal metric of constant $Q$ curvature.
This is analogous to the existence of constant Gauss curvature metrics in
dimension two. Let us consider the following functionals%
\begin{equation}
I\left( w\right) =\int_{M}\left\vert W\right\vert ^{2}wd\mu -\frac{1}{4}%
\left( \int_{M}\left\vert W\right\vert ^{2}d\mu \right) \log \left( \frac{1}{%
\mu \left( M\right) }\int_{M}e^{4w}d\mu \right) ,  \label{eq2.4}
\end{equation}%
\begin{eqnarray}
&&II\left( w\right)  \label{eq2.5} \\
&=&\int_{M}Pw\cdot wd\mu +2\int_{M}Qwd\mu -\frac{1}{2}\left( \int_{M}Qd\mu
\right) \log \left( \frac{1}{\mu \left( M\right) }\int_{M}e^{4w}d\mu \right)
,  \notag
\end{eqnarray}%
and%
\begin{equation}
III\left( w\right) =\int_{M}J_{e^{2w}g}^{2}d\mu _{e^{2w}g}-\int_{M}J^{2}d\mu
.  \label{eq2.6}
\end{equation}%
The Euler-Lagrange equation of functional $II$ is given by%
\begin{equation}
Pw+Q-\frac{\int_{M}Qd\mu }{\int_{M}e^{4w}d\mu }e^{4w}=0.  \label{eq2.7}
\end{equation}%
Or in another word,%
\begin{equation}
Q_{e^{2w}g}=const.  \label{eq2.8}
\end{equation}%
On the other hand, the Euler-Lagrange equation for functional $III$ is%
\begin{equation}
-\Delta _{e^{2w}g}J_{e^{2w}g}=0.  \label{eq2.9}
\end{equation}%
In \cite{CY} the general functional $F=\gamma _{1}I+\gamma _{2}II+\gamma
_{3}III$ was studied.

\begin{theorem}[\protect\cite{CY}]
\label{thm2.4}If the functional $F$ satisfies%
\begin{equation}
\gamma _{2}>0,\quad \gamma _{3}>0  \label{eq2.10}
\end{equation}%
and%
\begin{equation}
\kappa =\frac{\gamma _{1}}{2}\int_{M}\left\vert W\right\vert ^{2}d\mu
+\gamma _{2}\int_{M}Qd\mu <16\pi ^{2}\gamma _{2},  \label{eq2.11}
\end{equation}%
then there exists a minimizer for%
\begin{equation}
\inf_{w\in H^{2}\left( M\right) }F\left( w\right) .  \label{eq2.12}
\end{equation}%
Any minimizer must be smooth. If $w$ is a minimizer and we write $\widetilde{%
g}=e^{2w}g$, then%
\begin{equation}
\frac{\gamma _{1}}{2}\left\vert \widetilde{W}\right\vert _{\widetilde{g}%
}^{2}+\gamma _{2}\widetilde{Q}-\gamma _{3}\widetilde{\Delta }\widetilde{J}=%
\frac{\kappa }{\widetilde{\mu }\left( M\right) }.  \label{eq2.13}
\end{equation}%
Moreover for any $\varphi \in H^{2}\left( M\right) $ with%
\begin{equation}
\int_{M}\varphi d\widetilde{\mu }=0,  \label{eq2.14}
\end{equation}%
we have%
\begin{eqnarray}
&&\gamma _{2}\int_{M}\widetilde{P}\varphi \cdot \varphi d\widetilde{\mu }%
+\gamma _{3}\int_{M}\left[ \left( \widetilde{\Delta }\varphi +\left\vert 
\widetilde{\nabla }\varphi \right\vert _{\widetilde{g}}^{2}\right) ^{2}-2%
\widetilde{J}\left\vert \widetilde{\nabla }\varphi \right\vert _{\widetilde{g%
}}^{2}\right] d\widetilde{\mu }  \label{eq2.15} \\
&\geq &\frac{\kappa }{2}\log \left( \frac{1}{\widetilde{\mu }\left( M\right) 
}\int_{M}e^{4\varphi }d\widetilde{\mu }\right) .  \notag
\end{eqnarray}%
Here $\int_{M}\widetilde{P}\varphi \cdot \varphi d\widetilde{\mu }$ is
understood in distribution sense.
\end{theorem}

For the functional $II$, we have a similar existence result.

\begin{theorem}[\protect\cite{CY}]
\label{thm2.5}If%
\begin{equation}
\kappa _{g}=\int_{M}Qd\mu <16\pi ^{2},  \label{eq2.16}
\end{equation}%
$P\geq 0$ and the kernel of $P$ consists only of constant functions, then%
\begin{equation}
\inf_{w\in H^{2}\left( M\right) }II\left( w\right)  \label{eq2.17}
\end{equation}%
is achieved. Any minimizer must be smooth. If $w$ is a minimizer and we
write $\widetilde{g}=e^{2w}g$, then%
\begin{equation}
\widetilde{Q}=\frac{\kappa _{g}}{\widetilde{\mu }\left( M\right) }.
\label{eq2.18}
\end{equation}%
Moreover for any $\varphi \in H^{2}\left( M\right) $ with%
\begin{equation}
\int_{M}\varphi d\widetilde{\mu }=0,  \label{eq2.19}
\end{equation}%
we have%
\begin{equation}
\int_{M}\widetilde{P}\varphi \cdot \varphi d\widetilde{\mu }\geq \frac{%
\kappa _{g}}{2}\log \left( \frac{1}{\widetilde{\mu }\left( M\right) }%
\int_{M}e^{4\varphi }d\widetilde{\mu }\right) .  \label{eq2.20}
\end{equation}%
Here $\int_{M}\widetilde{P}\varphi \cdot \varphi d\widetilde{\mu }$ is
understood in distribution sense.
\end{theorem}

More results on the existence of conformal metrics with constant $Q$
curvature can be found in \cite{DM}. The main ingredient for Theorem \ref%
{thm2.4} and \ref{thm2.5} is the following version of Adams inequality (\cite%
{Ad}):

\begin{theorem}[\protect\cite{BCY,F}]
\label{thm2.6}Let $\left( M,g\right) $ be a smooth compact $4$-dimensional
Riemannian manifold with $P\geq 0$ and kernel of $P$ consists only of
constant functions, then for any $w\in H^{2}\left( M\right) $ with%
\begin{equation}
\int_{M}wd\mu =0,  \label{eq2.21}
\end{equation}%
we have%
\begin{equation}
\int_{M}\exp \left( 32\pi ^{2}\frac{w^{2}}{\int_{M}Pw\cdot wd\mu }\right)
d\mu \leq c\left( M,g\right) <\infty .  \label{eq2.22}
\end{equation}%
In particular%
\begin{equation}
\log \left( \frac{1}{\mu \left( M\right) }\int_{M}e^{4w}d\mu \right) \leq 
\frac{1}{8\pi ^{2}}\int_{M}Pw\cdot wd\mu +c\left( M,g\right) .
\label{eq2.23}
\end{equation}%
Here $\int_{M}Pw\cdot wd\mu $ is understood in distribution sense.
\end{theorem}

Adams inequality was discovered in \cite{Ad} with the motivation of
simplifying the original proof in \cite{M}. In particular a higher order
sharp inequality was derived through the O'Neil inequality for convolution
operator (see \cite{O}) and an one dimensional calculus lemma due to
Adams-Garsia. Theorem \ref{thm2.6} can be proven by modifying O'Neil
inequality and the calculus lemma.

For some geometrical and topological applications of these related equations
we refer the readers to \cite{CGY1, CGY2, G1}.

\section{Dimension at least $5$\label{sec3}}

The analysis of $Q$ curvature and Paneitz operator in dimension greater than 
$4$ has some similarity to the analysis of scalar curvature and conformal
Laplacian operator in dimension greater than $2$. The research related to
Yamabe problem serves as a nice model for asking interesting questions in
the study of Paneitz operator. However due to the fact second order
differential equations are much better understood than higher order
differential equations, sometime the analogous problem for $Q$ curvature can
be more challenging.

Based on the fact the first eigenfunction of conformal Laplacian operator
can always be chosen as positive everywhere, it was observed in \cite{KW}
that in a fixed conformal class, we can always find a metric whose scalar
curvature is only of one sign i.e. the scalar curvature is either strictly
positive, or identically zero, or strictly negative.

\begin{problem}
\label{prob3.1}Let $\left( M,g\right) $ be a smooth compact Riemannian
manifold with dimension $n\geq 5$, can we always find a conformal metric $%
\widetilde{g}$ such that $\widetilde{Q}$ is either strictly positive, or
identically zero, or strictly negative?
\end{problem}

This seems to be a difficult question. One of the obstacle is fourth order
symmetric elliptic operators can have no positive first eigenfunction at
all. Indeed let $M$ be any smooth compact Riemannian manifold, $\lambda $ be
the smallest positive eigenvalue of $-\Delta $, then the first eigenfunction
of $\left( -\Delta \right) ^{2}+2\lambda \Delta $ must change sign. Though
the answer to Problem \ref{prob3.1} remains mysterious, partial solution to
a related problem was found recently in \cite{HY4}. Recall on a smooth
compact Riemannian manifold $\left( M,g\right) $ with dimension greater than 
$2$, we have%
\begin{equation*}
\exists \widetilde{g}\in \left[ g\right] \text{ with }\widetilde{R}%
>0\Longleftrightarrow \lambda _{1}\left( L_{g}\right) >0.
\end{equation*}%
Here $\left[ g\right] $ denotes the conformal class of metrics associated
with $g$. The same statement remains true if we replace "$>$" by "$<$" or "$%
= $" (see \cite{LP}). It is worth pointing out the sign of $\lambda
_{1}\left( L_{g}\right) $ is a conformal invariant. In particular the above
statement gives a conformal invariant condition which is equivalent to the
existence of a conformal metric with positive scalar curvature.

\begin{problem}
\label{prob3.2}Let $\left( M,g\right) $ be a smooth compact Riemannian
manifold with dimension $n\geq 5$, can we find a conformal invariant
condition which is equivalent to the existence of a conformal metric with
positive $Q$ curvature? Same questions can be asked when "positive" is
replaced by "negative" or "zero".
\end{problem}

\cite{HY4} gives a partial answer to this problem under the assumption the
Yamabe invariant $Y\left( g\right) >0$.

\begin{theorem}[\protect\cite{HY4}]
\label{thm3.1}Let $n\geq 5$ and $\left( M^{n},g\right) $ be a smooth compact
Riemannian manifold with Yamabe invariant $Y\left( g\right) >0$, then the
following statements are equivalent

\begin{enumerate}
\item $\exists \widetilde{g}\in \left[ g\right] $ with $\widetilde{Q}>0$.

\item $\ker P_{g}=0$ and the Green's function of Paneitz operator $%
G_{P}\left( p,q\right) >0$ for any $p,q\in M,p\neq q$.

\item $\ker P_{g}=0$ and there exists a $p\in M$ such that $G_{P}\left(
p,q\right) >0$ for $q\in M\backslash \left\{ p\right\} $.
\end{enumerate}
\end{theorem}

By transformation law (\ref{eq1.7}) we know $\ker P_{g}=0$ is a conformal
invariant condition, moreover under this assumption, the Green's functions
of Paneitz operator $G_{P}$ satisfy%
\begin{equation}
G_{P,\rho ^{\frac{4}{n-4}}g}\left( p,q\right) =\rho \left( p\right)
^{-1}\rho \left( q\right) ^{-1}G_{P,g}\left( p,q\right) .  \label{eq3.1}
\end{equation}%
In particular, the fact $G_{P}>0$ is also a conformal invariant condition.
Of course this condition is clearly more complicated than the one given for
the scalar curvature case, however the main strength of Theorem \ref{thm3.1}
lies in that it gives an easy to check necessary and sufficient condition
for the positivity of the Green's function of Paneitz operator for metrics
of positive Yamabe class. As we will see shortly, the positivity of Green's
function is crucial in the study of $Q$ curvature equation.

The main ingredients in proof of Theorem \ref{thm3.1} is an identity similar
to (\ref{eq2.3}) in higher dimension.

\begin{theorem}[\protect\cite{HY4}]
\label{thm3.2}Assume $n\geq 5$, $\left( M^{n},g\right) $ is a smooth compact
Riemannian manifold with $Y\left( g\right) >0$, $p\in M$, then we have $%
G_{L,p}^{\frac{n-4}{n-2}}\left\vert Rc_{G_{L,p}^{\frac{4}{n-2}}g}\right\vert
_{g}^{2}\in L^{1}\left( M\right) $ and%
\begin{equation}
P\left( G_{L,p}^{\frac{n-4}{n-2}}\right) =c_{n}\delta _{p}-\frac{n-4}{\left(
n-2\right) ^{2}}G_{L,p}^{\frac{n-4}{n-2}}\left\vert Rc_{G_{L,p}^{\frac{4}{n-2%
}}g}\right\vert _{g}^{2}  \label{eq3.2}
\end{equation}%
in distribution sense. Here%
\begin{equation}
c_{n}=2^{-\frac{n-6}{n-2}}n^{\frac{2}{n-2}}\left( n-1\right) ^{-\frac{n-4}{%
n-2}}\left( n-2\right) \left( n-4\right) \omega _{n}^{\frac{2}{n-2}},
\label{eq3.3}
\end{equation}%
$\omega _{n}$ is the volume of unit ball in $\mathbb{R}^{n}$, $G_{L,p}$ is
the Green's function of conformal Laplacian operator with pole at $p$.
\end{theorem}

Here we will give another conformal invariant condition for the existence of
conformal metric with positive $Q$ curvature. To achieve this we first
introduce some notations.

Let $\left( M,g\right) $ be a smooth compact Riemannian manifold. If $%
K=K\left( p,q\right) $ is a suitable function on $M\times M$, we define an
operator $T_{K}$ as%
\begin{equation}
T_{K}\left( \varphi \right) \left( p\right) =\int_{M}K\left( p,q\right)
\varphi \left( q\right) d\mu \left( q\right)  \label{eq3.4}
\end{equation}%
for any nice function $\varphi $ on $M$. If $K^{\prime }=K^{\prime }\left(
p,q\right) $ is another function on $M\times M$, then we write%
\begin{equation}
\left( K\ast K^{\prime }\right) \left( p,q\right) =\int_{M}K\left(
p,s\right) K^{\prime }\left( s,q\right) d\mu \left( s\right) .  \label{eq3.5}
\end{equation}

If $n\geq 5$ and $Y\left( g\right) >0$, we write 
\begin{eqnarray}
&&H\left( p,q\right)  \label{eq3.6} \\
&=&2^{\frac{n-6}{n-2}}n^{-\frac{2}{n-2}}\left( n-1\right) ^{\frac{n-4}{n-2}%
}\left( n-2\right) ^{-1}\left( n-4\right) ^{-1}\omega _{n}^{-\frac{2}{n-2}%
}G_{L}\left( p,q\right) ^{\frac{n-4}{n-2}},  \notag
\end{eqnarray}%
and%
\begin{eqnarray}
&&\Gamma _{1}\left( p,q\right)  \label{eq3.7} \\
&=&2^{\frac{n-6}{n-2}}n^{-\frac{2}{n-2}}\left( n-1\right) ^{\frac{n-4}{n-2}%
}\left( n-2\right) ^{-3}\omega _{n}^{-\frac{2}{n-2}}G_{L}\left( p,q\right) ^{%
\frac{n-4}{n-2}}\left\vert Rc_{G_{L,p}^{\frac{4}{n-2}}g}\right\vert
_{g}^{2}\left( q\right) .  \notag
\end{eqnarray}%
Then (\ref{eq3.2}) becomes%
\begin{equation}
P_{q}H\left( p,q\right) =\delta _{p}\left( q\right) -\Gamma _{1}\left(
p,q\right) .  \label{eq3.8}
\end{equation}%
Note that by the calculation in \cite[Section 2]{HY4},%
\begin{equation}
\Gamma _{1}\left( p,q\right) =O\left( \overline{pq}^{4-n}\right) ,
\label{eq3.9}
\end{equation}%
here $\overline{pq}$ denotes the distance between $p$ and $q$. Assume for
all $p\in M$,%
\begin{equation}
0\leq \int_{M}\Gamma _{1}\left( p,q\right) d\mu \left( q\right) \leq \alpha
<\infty ,  \label{eq3.10}
\end{equation}%
then%
\begin{equation}
\left\Vert T_{\Gamma _{1}}\varphi \right\Vert _{L^{\infty }\left( M\right)
}\leq \alpha \left\Vert \varphi \right\Vert _{L^{\infty }\left( M\right) }.
\label{eq3.11}
\end{equation}%
Moreover if we let $\widetilde{g}=\rho ^{\frac{4}{n-4}}g$, here $\rho $ is a
positive smooth function, then for any smooth function $\varphi $ on $M$,%
\begin{equation}
T_{\widetilde{\Gamma }_{1}}\left( \varphi \right) =\rho ^{-1}T_{\Gamma
_{1}}\left( \rho \varphi \right) .  \label{eq3.12}
\end{equation}%
In another word, $T_{\widetilde{\Gamma }_{1}}$ is similar to $T_{\Gamma
_{1}} $. Hence they have the same spectrum and spectral radius i.e. $\sigma
\left( T_{\widetilde{\Gamma }_{1}}\right) =\sigma \left( T_{\Gamma
_{1}}\right) $ and $r_{\sigma }\left( T_{\widetilde{\Gamma }_{1}}\right)
=r_{\sigma }\left( T_{\Gamma _{1}}\right) $ (the spectral radius).

\begin{theorem}
\label{thm3.3}Assume $n\geq 5$, $\left( M^{n},g\right) $ is a smooth compact
Riemannian manifold with $Y\left( g\right) >0$, then%
\begin{equation*}
\exists \widetilde{g}\in \left[ g\right] \text{ with }\widetilde{Q}%
>0.\Longleftrightarrow \text{the spectral radius }r_{\sigma }\left(
T_{\Gamma _{1}}\right) <1.
\end{equation*}%
Moreover if $r_{\sigma }\left( T_{\Gamma _{1}}\right) <1$, then $\ker P=0$
and%
\begin{equation}
G_{P}=H+\sum_{k=1}^{\infty }\Gamma _{k}\ast H,  \label{eq3.13}
\end{equation}%
here%
\begin{equation}
\Gamma _{k}=\Gamma _{1}\ast \cdots \ast \Gamma _{1}\text{ (}k\text{ times),}
\label{eq3.14}
\end{equation}%
$H$ and $\Gamma _{1}$ are given in (\ref{eq3.6}) and (\ref{eq3.7}). The
convergence in (\ref{eq3.13}) is uniform in the sense that%
\begin{equation*}
G_{P}-H-\sum_{k=1}^{l}\Gamma _{k}\ast H\longrightarrow 0
\end{equation*}%
uniformly on $M\times M$ as $l\rightarrow \infty $. In particular, $%
G_{P}\geq H$, moreover if $G_{P}\left( p,q\right) =H\left( p,q\right) $ for
some $p\neq q$, then $\left( M,g\right) $ is conformal equivalent to the
standard $S^{n}$.
\end{theorem}

\begin{proof}
Assume there exists a $\widetilde{g}\in \left[ g\right] $ with $\widetilde{Q}%
>0$, then we hope to show $r_{\sigma }\left( T_{\Gamma _{1}}\right) <1$.
Because $r_{\sigma }\left( T_{\Gamma _{1}}\right) =r_{\sigma }\left( T_{%
\widetilde{\Gamma }_{1}}\right) $, replacing $g$ with $\widetilde{g}$ we can
assume the background metric satisfies $Q>0$. By (\ref{eq3.8}) we know for
any smooth function $\varphi $,%
\begin{equation}
\varphi =T_{H}\left( P\varphi \right) +T_{\Gamma _{1}}\left( \varphi \right)
.  \label{eq3.15}
\end{equation}%
Taking $\varphi =1$ in (\ref{eq3.15}) we get%
\begin{equation}
\int_{M}\Gamma _{1}\left( p,q\right) d\mu \left( q\right) =1-\frac{n-4}{2}%
\int_{M}H\left( p,q\right) Q\left( q\right) d\mu \left( q\right) .
\label{eq3.16}
\end{equation}%
Using the fact $Q>0$ we know there exists a constant $\alpha $ such that%
\begin{equation*}
\int_{M}\Gamma _{1}\left( p,q\right) d\mu \left( q\right) \leq \alpha <1
\end{equation*}%
for all $p\in M$. It follows that%
\begin{equation*}
\left\Vert T_{\Gamma _{1}}\right\Vert _{\mathcal{L}\left( L^{\infty
},L^{\infty }\right) }\leq \alpha
\end{equation*}%
and hence%
\begin{equation*}
r_{\sigma }\left( T_{\Gamma _{1}}\right) \leq \alpha <1.
\end{equation*}

On the other hand, assume $r_{\sigma }\left( T_{\Gamma _{1}}\right) <\alpha
<1$, then we can find a constant $k_{0}$ such that for $k\geq k_{0}$,%
\begin{equation*}
\left\Vert T_{\Gamma _{k}}\right\Vert _{\mathcal{L}\left( L^{\infty
},L^{\infty }\right) }<\alpha ^{k}.
\end{equation*}%
It follows that%
\begin{equation*}
\int_{M}\Gamma _{k}\left( p,q\right) d\mu \left( q\right) <\alpha ^{k}.
\end{equation*}%
Fix $m>\frac{n}{4}$, using estimate (\ref{eq3.9}) we see for all $k\geq
k_{0}+m$,%
\begin{equation*}
\left\Vert \Gamma _{k}\right\Vert _{L^{\infty }}\leq \alpha ^{k-m}\left\Vert
\Gamma _{m}\right\Vert _{L^{\infty }}\leq c\alpha ^{k}.
\end{equation*}%
In particular $\left\Vert \Gamma _{k}\right\Vert _{L^{\infty }}\rightarrow 0$
and%
\begin{equation*}
\left\Vert \Gamma _{k}\ast H\right\Vert _{L^{\infty }}\leq c\left\Vert
\Gamma _{k}\right\Vert _{L^{\infty }}\leq c\alpha ^{k}.
\end{equation*}%
Iterating (\ref{eq3.15}) we see%
\begin{equation*}
\varphi =T_{H+\Gamma _{1}\ast H+\cdots +\Gamma _{k-1}\ast H}\left( P\varphi
\right) +T_{\Gamma _{k}}\left( \varphi \right) .
\end{equation*}%
Let $k\rightarrow \infty $, we see%
\begin{equation*}
\varphi =T_{H+\sum_{k=1}^{\infty }\Gamma _{k}\ast H}\left( P\varphi \right) .
\end{equation*}%
In particular, $P\varphi =0$ implies $\varphi =0$ i.e. $\ker P=0$. Moreover%
\begin{equation*}
G_{P}=H+\sum_{k=1}^{\infty }\Gamma _{k}\ast H.
\end{equation*}%
In particular $G_{P}\geq H>0$. If $G_{P}\left( p,q\right) =H\left(
p,q\right) $ for some $p\neq q$, then $\Gamma _{1}\left( p,\cdot \right) =0$%
, in another word%
\begin{equation*}
Rc_{G_{L,p}^{\frac{4}{n-2}}g}=0.
\end{equation*}%
Since $\left( M\backslash \left\{ p\right\} ,G_{L,p}^{\frac{4}{n-2}}g\right) 
$ is asymptotically flat, it follows from relative volume comparison theorem
that $\left( M\backslash \left\{ p\right\} ,G_{L,p}^{\frac{4}{n-2}}g\right) $
is isometric to $\mathbb{R}^{n}$, hence $\left( M,g\right) $ is conformal
equivalent to standard $S^{n}$.

Since $G_{P}>0$, it follows from Theorem \ref{thm3.1} that there exists $%
\widetilde{g}\in \left[ g\right] $ with $\widetilde{Q}>0$.
\end{proof}

We remark that the infinite series expansion of $G_{P}$ in (\ref{eq3.13}) is
similar to those for Green's function of Laplacian in \cite{Au1}.

\begin{remark}
\label{rmk3.1}Indeed it follows from (\ref{eq3.16}) that as long as $Y\left(
g\right) >0$ and%
\begin{equation*}
\int_{M}H\left( p,q\right) Q\left( q\right) d\mu \left( q\right) >0
\end{equation*}%
for all $p\in M$, then $r_{\sigma }\left( T_{\Gamma _{1}}\right) <1$. In
particular this is the case when $Q\geq 0$ and not identically zero.
\end{remark}

\begin{problem}
\label{prob3.3}Let $\left( M,g\right) $ be a smooth compact Riemannian
manifold with dimension $n\geq 5$, can we find a metric $\widetilde{g}\in %
\left[ g\right] $ such that $\widetilde{Q}=const$?
\end{problem}

This turns out to be a difficult problem with only partial solutions
available. If we write the unknown metric $\widetilde{g}=\rho ^{\frac{4}{n-4}%
}g$, then we need to solve%
\begin{equation}
P\rho =const\cdot \rho ^{\frac{n+4}{n-4}},\quad \rho \in C^{\infty }\left(
M\right) ,\rho >0.  \label{eq3.17}
\end{equation}%
As in the case of Yamabe problem, (\ref{eq3.17}) has a variational
structure. Indeed, for $u\in C^{\infty }\left( M\right) $, let%
\begin{eqnarray}
E\left( u\right) &=&\int_{M}Pu\cdot ud\mu  \label{eq3.18} \\
&=&\int_{M}\left[ \left( \Delta u\right) ^{2}-4A\left( \nabla u,\nabla
u\right) +\left( n-2\right) J\left\vert \nabla u\right\vert ^{2}+\frac{n-4}{2%
}Qu^{2}\right] d\mu .  \notag
\end{eqnarray}%
Clearly we can extend $E\left( u\right) $ continuously to $u\in H^{2}\left(
M\right) $. Let%
\begin{equation}
Y_{4}\left( g\right) =\inf_{u\in H^{2}\left( M\right) \backslash \left\{
0\right\} }\frac{E\left( u\right) }{\left\Vert u\right\Vert _{L^{\frac{2n}{%
n-4}}}^{2}},  \label{eq3.19}
\end{equation}%
then $Y_{4}\left( g\right) $ is a conformal invariant in the same spirit as $%
Y\left( g\right) $. If $Y_{4}\left( g\right) $ is achieved at a smooth
positive function $\rho $, then it satisfies (\ref{eq3.17}). On the other
hand, even if $Y_{4}\left( g\right) $ is achieved at a function $u\in
H^{2}\left( M\right) $, we can not conclude whether $u$ changes sign or not.
An observation made in \cite{R} says that if $P>0$ and $G_{P}>0$, then the
minimizer must be smooth and either strictly positive or strictly negative.
We remark that it had been observed in \cite{HeR1, HeR2, HuR} that the
positivity of Green's function of Paneitz operator plays crucial roles in
various issues related to $Q$ curvature. Without the classical maximum
principle, it is hard to know the sign of Green's function of the fourth
order operator. A breakthrough was made in \cite{GM}, which provides an easy
to check sufficient condition for the positivity of Green's function.

\begin{theorem}[\protect\cite{GM}]
\label{thm3.4}Assume $n\geq 5$, $\left( M^{n},g\right) $ is a smooth compact
Riemannian manifold with $R>0$, $Q\geq 0$ and not identically zero, then $%
P>0 $. Moreover if $u$ is a nonzero smooth function with $Pu\geq 0$, then $%
u>0$ and $R_{u^{\frac{4}{n-4}}g}>0$. In particular, $G_{P}>0$.
\end{theorem}

Note that the necessary and sufficient condition in Theorem \ref{thm3.1} is
motivated by \cite{GM, HuR}. The final solution of Yamabe problem uses the
positive mass theorem (see \cite{LP, S}). The corresponding statement for
the Paneitz operator is established in \cite{GM, HuR}. Indeed an elementary
but ingenious calculation in \cite{HuR} justifies the positivity of mass
under the assumption of positivity of Green's function of Paneitz operator
for locally conformally flat manifolds. As pointed out in \cite{GM}, the
same calculation carries through to nonlocally conformally flat manifolds in
dimension $5,6$ and $7$ as well. A close connection between the positive
mass result and formula (\ref{eq3.2}) is found in \cite[section 6]{HY4}.
Combine these with Theorem \ref{thm3.1} and \ref{thm3.3} we have

\begin{theorem}[\protect\cite{GM, HY4, HY5, HuR}]
\label{thm3.5}Assume $n\geq 5$, $\left( M^{n},g\right) $ is a smooth compact
Riemannian manifold with $Y\left( g\right) >0$ and the spectral radius $%
r_{\sigma }\left( T_{\Gamma _{1}}\right) <1$ ($\Gamma _{1}$ is given by (\ref%
{eq3.7})). If $n=5,6,7$ or $\left( M,g\right) $ is locally conformally flat
near $p\in M$, then $\ker P=0$ and under conformal normal coordinate at $p$, 
$x_{1},\cdots ,x_{n}$,%
\begin{equation*}
G_{P,p}=\frac{1}{2n\left( n-2\right) \left( n-4\right) \omega _{n}}\left(
r^{4-n}+A+O\left( r\right) \right) ,
\end{equation*}%
with the constant $A\geq 0$, here $r=\left\vert x\right\vert $, $\omega _{n}$
is the volume of the unit ball in $\mathbb{R}^{n}$. Moreover $A=0$ if and
only if $\left( M,g\right) $ is conformal equivalent to $S^{n}$.
\end{theorem}

Indeed following \cite[Section 6]{HY4} we note that under the assumption of
Theorem \ref{thm3.5} (see \cite{LP})%
\begin{equation*}
G_{L,p}=\frac{1}{4n\left( n-1\right) \omega _{n}}\left( r^{2-n}+O\left(
r^{-1}\right) \right) .
\end{equation*}%
Let $H_{p}\left( q\right) =H\left( p,q\right) $, then%
\begin{equation*}
G_{P,p}-H_{p}=\frac{A}{2n\left( n-2\right) \left( n-4\right) \omega _{n}}%
+O\left( r\right) .
\end{equation*}%
It follows from (\ref{eq3.8}) that%
\begin{equation*}
P\left( G_{P,p}-H_{p}\right) \left( q\right) =\Gamma _{1}\left( p,q\right) .
\end{equation*}%
Hence%
\begin{eqnarray*}
&&A \\
&=&2n\left( n-2\right) \left( n-4\right) \omega _{n}\int_{M}G_{P}\left(
p,q\right) \Gamma _{1}\left( p,q\right) d\mu \left( q\right) \\
&=&2^{\frac{2\left( n-4\right) }{n-2}}n^{\frac{n-4}{n-2}}\left( n-1\right) ^{%
\frac{n-4}{n-2}}\left( n-2\right) ^{-2}\left( n-4\right) \omega _{n}^{\frac{%
n-4}{n-2}}\int_{M}G_{P,p}G_{L,p}^{\frac{n-4}{n-2}}\left\vert Rc_{G_{L,p}^{%
\frac{4}{n-2}}g}\right\vert _{g}^{2}d\mu .
\end{eqnarray*}%
This is exactly the formula proven in \cite{HuR}. Theorem \ref{thm3.5}
follows from this calculation. With Theorem \ref{thm3.1}, \ref{thm3.3} and %
\ref{thm3.5} at hand, we are able to give the first partial solution to
Problem \ref{prob3.3}.

\begin{theorem}[\protect\cite{GM, HY5}]
\label{thm3.6}Let $\left( M,g\right) $ be a smooth compact $n$ dimensional
Riemannian manifold with $n\geq 5$, $Y\left( g\right) >0$, $Y_{4}\left(
g\right) >0$, $r_{\sigma }\left( T_{\Gamma _{1}}\right) <1$, then

\begin{enumerate}
\item $Y_{4}\left( g\right) \leq Y_{4}\left( S^{n}\right) $, and equality
holds if and only if $\left( M,g\right) $ is conformally diffeomorphic to
the standard sphere.

\item $Y_{4}\left( g\right) $ is always achieved. Any minimizer must be
smooth and cannot change sign. In particular we can find a constant $Q$
curvature metric in the conformal class.

\item If $\left( M,g\right) $ is not conformally diffeomorphic to the
standard sphere, then the set of all minimizers $u$ for $Y_{4}\left(
g\right) $, after normalizing with $\left\Vert u\right\Vert _{L^{\frac{2n}{%
n-4}}}=1$, is compact in $C^{\infty }$ topology.
\end{enumerate}
\end{theorem}

It is worth pointing out that for a locally conformally flat manifold with
positive Yamabe invariant and Poincare exponent less than $\frac{n-4}{2}$
(see \cite{SY}), Theorem \ref{thm3.6} was proved in \cite{QR2} by apriori
estimates (using method of moving planes for integral equations developed in 
\cite{CLO}) and connecting the equation to Yamabe equation through a path of
integral equations.

Note that $Y_{4}\left( g\right) >0$ is the same thing as $P>0$. Either one
of the following conditions guarantee the positivity of Paneitz operator

\begin{itemize}
\item \cite{GM, XY1}: $n\geq 5$, $R>0$, $Q\geq 0$ and not identically zero;

\item \cite[Theorem 1.6]{CHY}: $n\geq 5$, $J\geq 0$, $\sigma _{2}\left(
A\right) \geq 0$ and $\left( M,g\right) $ is not Ricci flat.
\end{itemize}

In applications we are usually interested in metrics not just with $Q>0$,
but with both $R>0$ and $Q>0$. This leads us to a question similar to
Problem \ref{prob3.2}.

\begin{problem}[{\protect\cite[Problem 1.1]{GHL}}]
\label{prob3.4}For a smooth compact Riemannian manifold with dimension at
least $5$, can we find a conformal invariant condition which is equivalent
to the existence of a conformal metric with positive scalar and $Q$
curvature?
\end{problem}

\begin{theorem}[\protect\cite{GHL}]
\label{thm3.7}Let $\left( M,g\right) $ be a smooth compact Riemannian
manifold with dimension $n\geq 6$. Denote%
\begin{equation*}
Y_{4}^{+}\left( g\right) =\frac{n-4}{2}\inf_{\widetilde{g}\in \left[ g\right]
}\frac{\int_{M}\widetilde{Q}d\widetilde{\mu }}{\left( \widetilde{\mu }\left(
M\right) \right) ^{\frac{n-4}{n}}}=\inf_{\substack{ u\in C^{\infty }\left(
M\right)  \\ u>0}}\frac{\int_{M}Pu\cdot ud\mu }{\left\Vert u\right\Vert _{L^{%
\frac{2n}{n-4}}}^{2}}.
\end{equation*}%
and%
\begin{equation*}
Y_{4}^{\ast }\left( g\right) =\frac{n-4}{2}\inf_{\substack{ \widetilde{g}\in %
\left[ g\right]  \\ \widetilde{R}>0}}\frac{\int_{M}\widetilde{Q}d\widetilde{%
\mu }}{\left( \widetilde{\mu }\left( M\right) \right) ^{\frac{n-4}{n}}}.
\end{equation*}%
If $Y\left( g\right) >0$ and $Y_{4}^{\ast }\left( g\right) >0$, then there
exists a metric $\widetilde{g}\in \left[ g\right] $ satisfying $\widetilde{R}%
>0$ and $\widetilde{Q}>0$. In particular, $P>0$, the Green's function $%
G_{P}>0$, and $Y_{4}\left( g\right) $ is achieved at a positive smooth
function $u$ with $R_{u^{\frac{4}{n-4}}g}>0$ and $Q_{u^{\frac{4}{n-4}%
}g}=const$. Moreover,%
\begin{equation*}
Y_{4}\left( g\right) =Y_{4}^{+}\left( g\right) =Y_{4}^{\ast }\left( g\right)
.
\end{equation*}
\end{theorem}

\begin{corollary}[\protect\cite{GHL}]
\label{cor3.1}Let $\left( M,g\right) $ be a smooth compact Riemannian
manifold with dimension $n\geq 6$. Then the following statements are
equivalent

\begin{enumerate}
\item $Y\left( g\right) >0,P>0$.

\item $Y\left( g\right) >0,Y_{4}^{\ast }\left( g\right) >0$.

\item there exists a metric $\widetilde{g}\in \left[ g\right] $ satisfying $%
\widetilde{R}>0$ and $\widetilde{Q}>0$.
\end{enumerate}
\end{corollary}

Corollary \ref{cor3.1} answers Problem \ref{prob3.4} for dimension at least $%
6$. It also tells us in Theorem \ref{thm3.6}, condition $r_{\sigma }\left(
T_{\Gamma _{1}}\right) <1$ is implied by the positivity of $Y\left( g\right) 
$ and $Y_{4}\left( g\right) $ when $n\geq 6$. The case $n=5$ still remains
open for Problem \ref{prob3.4}.

\begin{problem}
\label{prob3.5}Let $\left( M,g\right) $ be a smooth compact Riemannian
manifold with dimension $n\geq 5$, do we have%
\begin{equation*}
Y\left( g\right) >0,Q>0\Longrightarrow P>0?
\end{equation*}%
The answer is probably negative.
\end{problem}

This seems to be a subtle question. Indeed from \cite{GM, XY1}, we know when
both $R$ and $Q$ are positive, then $P$ is positive definite. If we have $%
Y\left( g\right) >0$ and $Q>0$ instead, then some conformal metrics have
positive scalar curvature. However the set of metrics with positive scalar
curvature may be disjoint with those with positive $Q$ curvature.
Nevertheless Theorem \ref{thm3.1} tells us $\ker P=0$ and $G_{P}>0$. In \cite%
{HY5}, it is shown this is enough to find a constant $Q$ curvature in the
conformal class. Together with Theorem \ref{thm3.3}, we have another partial
answer to Problem \ref{prob3.3}.

\begin{theorem}[\protect\cite{HY5}]
\label{thm3.8}Let $\left( M,g\right) $ be a smooth compact $n$ dimensional
Riemannian manifold with $n\geq 5$, $Y\left( g\right) >0$, $r_{\sigma
}\left( T_{\Gamma _{1}}\right) <1$, then $\ker P=0$, the Green's function of 
$P$ is positive and there exists a conformal metric $\widetilde{g}$ with $%
\widetilde{Q}=1$.
\end{theorem}

Note that if the answer to Problem \ref{prob3.5} is positive, then Theorem %
\ref{thm3.8} would follow from Theorem \ref{thm3.6}. Without knowing the
positivity of Paneitz operator, we can not use the minimization problem (\ref%
{eq3.19}) to find the constant $Q$ curvature metrics. A different approach
was developed in \cite{HY5}. Under the assumption of Theorem \ref{thm3.8},
it follows from Theorem \ref{thm3.3} that $\ker P=0$ and $G_{P}>0$. If we
denote $f=\rho ^{\frac{n+4}{n-4}}$, then equation (\ref{eq3.17}) becomes%
\begin{equation}
T_{G_{P}}f=\frac{2}{n-4}f^{\frac{n-4}{n+4}},\quad f\in C^{\infty }\left(
M\right) ,f>0.  \label{eq3.20}
\end{equation}%
Let%
\begin{equation}
\Theta _{4}\left( g\right) =\sup_{f\in L^{\frac{2n}{n+4}}\left( M\right)
\backslash \left\{ 0\right\} }\frac{\int_{M}T_{G_{P}}f\cdot fd\mu }{%
\left\Vert f\right\Vert _{L^{\frac{2n}{n+4}}}^{2}}.  \label{eq3.21}
\end{equation}%
By (\ref{eq3.1}), we know $\Theta _{4}\left( g\right) $ is a conformal
invariant, moreover it has a nice geometrical description, which is local,
(see \cite[Section 2.1]{HY5})%
\begin{eqnarray}
\Theta _{4}\left( g\right) &=&\frac{2}{n-4}\sup \left\{ \frac{\int_{M}%
\widetilde{Q}d\widetilde{\mu }}{\left\Vert \widetilde{Q}\right\Vert _{L^{%
\frac{2n}{n+4}}\left( M,d\widetilde{\mu }\right) }^{2}}:\widetilde{g}\in %
\left[ g\right] \right\}  \label{eq3.22} \\
&=&\sup_{u\in W^{4,\frac{2n}{n+4}}\left( M\right) \backslash \left\{
0\right\} }\frac{\int_{M}Pu\cdot ud\mu }{\left\Vert Pu\right\Vert _{L^{\frac{%
2n}{n+4}}}^{2}}.  \notag
\end{eqnarray}%
It follows from the classical Hardy-Littlewood-Sobolev inequality $\Theta
_{4}\left( g\right) $ is always finite. The benefit of this formulation is
if $\Theta _{4}\left( g\right) $ is achieved by a maximizer $f$, we deduce
easily from the positivity of $G_{P}$ that $f$ cannot change sign. With
Theorem \ref{thm3.1}, \ref{thm3.3} and \ref{thm3.5} at hands, we have the
following statement about extremal problem for $\Theta _{4}\left( g\right) $:

\begin{theorem}[\protect\cite{HY5}]
\label{thm3.9}Assume $\left( M,g\right) $ is a smooth compact $n$
dimensional Riemannian manifold with $n\geq 5$, $Y\left( g\right) >0$, $%
r_{\sigma }\left( T_{\Gamma _{1}}\right) <1$, then

\begin{enumerate}
\item $\Theta _{4}\left( g\right) \geq \Theta _{4}\left( S^{n}\right) $,
here $S^{n}$ has the standard metric. $\Theta _{4}\left( g\right) =\Theta
_{4}\left( S^{n}\right) $ if and only if $\left( M,g\right) $ is conformally
diffeomorphic to the standard sphere.

\item $\Theta _{4}\left( g\right) $ is always achieved. Any maximizer $f$
must be smooth and cannot change sign. If $f>0$, then after scaling we have $%
G_{P}f=\frac{2}{n-4}f^{\frac{n-4}{n+4}}$ i.e. $Q_{f^{\frac{4}{n+4}}g}=1$.

\item If $\left( M,g\right) $ is not conformally diffeomorphic to the
standard sphere, then the set of all maximizers $f$ for $\Theta _{4}\left(
g\right) $, after normalizing with $\left\Vert f\right\Vert _{L^{\frac{2n}{%
n+4}}}=1$, is compact in the $C^{\infty }$ topology.
\end{enumerate}
\end{theorem}

The approach in Theorem \ref{thm3.9} is motivated from the integral
equations considered in \cite{HWY1, HWY2}. Integral equation formulation of
the $Q$ curvature equation had been used in \cite{QR2}.

At last we note that compactness problem for constant $Q$ curvature metrics
in a fixed conformal class has been considered in \cite{HeR1, L, LX, QR1, WZ}%
.

\section{Dimension $3$\label{sec4}}

As we will see soon, the analysis of $Q$ curvature equation in dimension $3$
is very different from those in dimension greater than $4$. On the other
hand, we expect the scalar curvature and $Q$ curvature plays more dominant
role for the geometry of the conformal class and the topology of the
underlying manifold in dimension $3$ than in dimension greater than $4$.
Because of this, we will list problems in dimension $3$ explicitly even
though some of them are similar to those in Section \ref{sec3}.

In dimension $3$, the $Q$ curvature is given by%
\begin{eqnarray}
Q &=&-\frac{1}{4}\Delta R-2\left\vert Rc\right\vert ^{2}+\frac{23}{32}R^{2}
\label{eq4.1} \\
&=&-\Delta J-2\left\vert A\right\vert ^{2}+\frac{3}{2}J^{2}  \notag \\
&=&-\Delta J+4\sigma _{2}\left( A\right) -\frac{1}{2}J^{2},  \notag
\end{eqnarray}%
here%
\begin{equation}
J=\frac{R}{4},\quad A=Rc-Jg.  \label{eq4.2}
\end{equation}%
The Paneitz operator is given by%
\begin{eqnarray}
P\varphi &=&\Delta ^{2}\varphi +4\func{div}\left[ Rc\left( \nabla \varphi
,e_{i}\right) e_{i}\right] -\frac{5}{4}\func{div}\left( R\nabla \varphi
\right) -\frac{1}{2}Q\varphi  \label{eq4.3} \\
&=&\Delta ^{2}\varphi +4\func{div}\left( A\left( \nabla \varphi
,e_{i}\right) e_{i}\right) -\func{div}\left( J\nabla \varphi \right) -\frac{1%
}{2}Q\varphi .  \notag
\end{eqnarray}%
Here $e_{1},e_{2},e_{3}$ is a local orthonormal frame with respect to $g$.
For any smooth positive function $\rho $,%
\begin{equation}
P_{\rho ^{-4}g}\varphi =\rho ^{7}P_{g}\left( \rho \varphi \right) .
\label{eq4.4}
\end{equation}%
Hence%
\begin{equation}
Q_{\rho ^{-4}g}=-2\rho ^{7}P_{g}\left( \rho \right) .  \label{eq4.5}
\end{equation}

\begin{problem}
\label{prob4.1}Let $\left( M,g\right) $ be a $3$ dimensional smooth compact
Riemannian manifold, can we always find a conformal metric $\widetilde{g}$
such that $\widetilde{Q}$ is either strictly positive, or identically zero,
or strictly negative? Can we find a conformal invariant condition which is
equivalent to the existence of a conformal metric with positive $Q$
curvature? Same questions can be asked when "positive" is replaced by
"negative" or "zero".
\end{problem}

Unfortunately this simple looking question only has partial solution at this
stage.

\begin{theorem}[\protect\cite{HY4}]
\label{thm4.1}Let $\left( M,g\right) $ be a smooth compact $3$ dimensional
Riemannian manifold with $Y\left( g\right) >0$, then the following
statements are equivalent:

\begin{enumerate}
\item $\exists \widetilde{g}\in \left[ g\right] $ with $\widetilde{Q}>0$.

\item $\ker P_{g}=0$ and the Green's function $G_{P}\left( p,q\right) <0$
for any $p,q\in M,p\neq q$.

\item $\ker P_{g}=0$ and there exists a $p\in M$ such that $G_{P}\left(
p,q\right) <0$ for $q\in M\backslash \left\{ p\right\} $.
\end{enumerate}
\end{theorem}

By transformation law (\ref{eq4.4}) we know $\ker P_{g}=0$ is a conformal
invariant condition. Under this assumption, the Green's functions satisfy%
\begin{equation}
G_{P,\rho ^{-4}g}\left( p,q\right) =\rho \left( p\right) ^{-1}\rho \left(
q\right) ^{-1}G_{P,g}\left( p,q\right) .  \label{eq4.6}
\end{equation}%
Hence the fact $G_{P}\left( p,q\right) <0$ for $p\neq q$ is a conformal
invariant condition. Theorem \ref{thm4.1} is based on the following identity:

\begin{theorem}[\protect\cite{HY4}]
\label{thm4.2}Let $\left( M,g\right) $ be a $3$ dimensional smooth compact
Riemannian manifold with $Y\left( g\right) >0$, $p\in M$, then we have $%
G_{L,p}^{-1}\left\vert Rc_{G_{L,p}^{4}g}\right\vert _{g}^{2}\in L^{1}\left(
M\right) $ and%
\begin{equation}
P\left( G_{L,p}^{-1}\right) =-256\pi ^{2}\delta _{p}+G_{L,p}^{-1}\left\vert
Rc_{G_{L,p}^{4}g}\right\vert _{g}^{2}  \label{eq4.7}
\end{equation}%
in distribution sense.
\end{theorem}

If $Y\left( g\right) >0$, we write 
\begin{equation}
H\left( p,q\right) =-\frac{G_{L}\left( p,q\right) ^{-1}}{256\pi ^{2}},
\label{eq4.8}
\end{equation}%
and%
\begin{equation}
\Gamma _{1}\left( p,q\right) =\frac{G_{L}\left( p,q\right) ^{-1}}{256\pi ^{2}%
}\left\vert Rc_{G_{L,p}^{4}g}\right\vert _{g}^{2}\left( q\right) .
\label{eq4.9}
\end{equation}%
Then (\ref{eq4.7}) becomes%
\begin{equation}
P_{q}H\left( p,q\right) =\delta _{p}\left( q\right) -\Gamma _{1}\left(
p,q\right) .  \label{eq4.10}
\end{equation}%
Note that by the calculation in \cite[Section 2]{HY4},%
\begin{equation}
\Gamma _{1}\left( p,q\right) =O\left( \overline{pq}^{-1}\right) ,
\label{eq4.11}
\end{equation}%
here $\overline{pq}$ denotes the distance between $p$ and $q$.

If we let $\widetilde{g}=\rho ^{-4}g$, here $\rho $ is a positive smooth
function, then for any smooth function $\varphi $ on $M$,%
\begin{equation}
T_{\widetilde{\Gamma }_{1}}\left( \varphi \right) =\rho ^{-1}T_{\Gamma
_{1}}\left( \rho \varphi \right) .  \label{eq4.12}
\end{equation}
Hence $T_{\widetilde{\Gamma }_{1}}$ and $T_{\Gamma _{1}}$ have the same
spectrum and spectral radius.

\begin{theorem}
\label{thm4.3}Let $\left( M,g\right) $ be a $3$ dimensional smooth compact
Riemannian manifold with $Y\left( g\right) >0$, then%
\begin{equation*}
\exists \widetilde{g}\in \left[ g\right] \text{ with }\widetilde{Q}%
>0.\Longleftrightarrow \text{the spectral radius }r_{\sigma }\left(
T_{\Gamma _{1}}\right) <1.
\end{equation*}%
Moreover if $r_{\sigma }\left( T_{\Gamma _{1}}\right) <1$, then $\ker P=0$
and%
\begin{equation}
G_{P}=H+\sum_{k=1}^{\infty }\Gamma _{k}\ast H,  \label{eq4.13}
\end{equation}%
here%
\begin{equation}
\Gamma _{k}=\Gamma _{1}\ast \cdots \ast \Gamma _{1}\text{ (}k\text{ times),}
\label{eq4.14}
\end{equation}%
$H$ and $\Gamma _{1}$ are given in (\ref{eq4.8}) and (\ref{eq4.9}). The
convergence in (\ref{eq4.13}) is uniform. In particular, $G_{P}\leq H$,
moreover if $G_{P}\left( p,q\right) =H\left( p,q\right) $ for some $p,q$,
then $\left( M,g\right) $ is conformal equivalent to the standard $S^{3}$.
\end{theorem}

\begin{proof}
The argument is basically same as the proof of Theorem \ref{thm3.3}. If
there exists a $\widetilde{g}\in \left[ g\right] $ with $\widetilde{Q}>0$,
by conformal invariance we can assume the background metric has positive $Q$
curvature. By (\ref{eq4.10}) for any smooth function $\varphi $,%
\begin{equation}
\varphi =T_{H}\left( P\varphi \right) +T_{\Gamma _{1}}\left( \varphi \right)
.  \label{eq4.15}
\end{equation}%
Taking $\varphi =1$ in (\ref{eq4.15}) we get%
\begin{equation}
\int_{M}\Gamma _{1}\left( p,q\right) d\mu \left( q\right) =1+\frac{1}{2}%
\int_{M}H\left( p,q\right) Q\left( q\right) d\mu \left( q\right) .
\label{eq4.16}
\end{equation}%
Hence for some $\alpha $%
\begin{equation*}
\int_{M}\Gamma _{1}\left( p,q\right) d\mu \left( q\right) \leq \alpha <1
\end{equation*}%
for all $p\in M$. It follows that%
\begin{equation*}
\left\Vert T_{\Gamma _{1}}\right\Vert _{\mathcal{L}\left( L^{\infty
},L^{\infty }\right) }\leq \alpha
\end{equation*}%
and%
\begin{equation*}
r_{\sigma }\left( T_{\Gamma _{1}}\right) \leq \alpha <1.
\end{equation*}

On the other hand, assume $r_{\sigma }\left( T_{\Gamma _{1}}\right) <\alpha
<1$, then we can find a constant $k_{0}$ such that for $k\geq k_{0}$,%
\begin{equation*}
\left\Vert T_{\Gamma _{k}}\right\Vert _{\mathcal{L}\left( L^{\infty
},L^{\infty }\right) }<\alpha ^{k}.
\end{equation*}%
It follows that%
\begin{equation*}
\int_{M}\Gamma _{k}\left( p,q\right) d\mu \left( q\right) <\alpha ^{k}.
\end{equation*}%
Using (\ref{eq4.11}) we see for all $k\geq k_{0}+2$,%
\begin{equation*}
\left\Vert \Gamma _{k}\right\Vert _{L^{\infty }}\leq \alpha ^{k-2}\left\Vert
\Gamma _{2}\right\Vert _{L^{\infty }}\leq c\alpha ^{k}.
\end{equation*}%
In particular $\left\Vert \Gamma _{k}\right\Vert _{L^{\infty }}\rightarrow 0$
and%
\begin{equation*}
\left\Vert \Gamma _{k}\ast H\right\Vert _{L^{\infty }}\leq c\left\Vert
\Gamma _{k}\right\Vert _{L^{\infty }}\leq c\alpha ^{k}.
\end{equation*}%
The remaining argument goes exactly the same as in the proof of Theorem \ref%
{thm3.3}.
\end{proof}

\begin{remark}
\label{rmk4.1}Indeed it follows from (\ref{eq4.16}) that as long as $Y\left(
g\right) >0$ and%
\begin{equation*}
\int_{M}H\left( p,q\right) Q\left( q\right) d\mu \left( q\right) <0
\end{equation*}%
for all $p\in M$, then $r_{\sigma }\left( T_{\Gamma _{1}}\right) <1$. In
particular this is the case when $Q\geq 0$ and not identically zero.
\end{remark}

It is worth pointing out that if $\ker P=0$, then because $\delta _{p}\in
H^{-2}\left( M\right) $, we see $G_{P,p}\in H^{2}\left( M\right) \subset C^{%
\frac{1}{2}}\left( M\right) $, in particular the Green's function has a
value at the pole, $G_{P,p}\left( p\right) $. This pole's value plays
exactly the same role as the mass for classical Yamabe problem. If $Y\left(
g\right) >0$, $r_{\sigma }\left( T_{\Gamma _{1}}\right) <1$ and $\left(
M,g\right) $ is not conformal diffeomorphic to the standard $S^{3}$, it
follows from Theorem \ref{thm4.3} that $G_{P}\left( p,q\right) <0$ for all $%
p,q\in M$. On the other hand, on the standard $S^{3}$, the Green's function
of Paneitz operator touches zero exactly at the pole and is negative away
from the pole.

\begin{problem}
\label{prob4.2}Let $\left( M,g\right) $ be a $3$ dimensional smooth compact
Riemannian manifold, can we find a metric $\widetilde{g}\in \left[ g\right] $
such that $\widetilde{Q}=const$?
\end{problem}

\begin{theorem}[\protect\cite{HY3, HY4}]
\label{thm4.4}Let $\left( M,g\right) $ be a $3$ dimensional smooth compact
Riemannian manifold with $Y\left( g\right) >0$ and $r_{\sigma }\left(
T_{\Gamma _{1}}\right) <1$, then there exists $\widetilde{g}\in \left[ g%
\right] $ such that $\widetilde{Q}=1$. Moreover as long as $\left(
M,g\right) $ is not conformal diffeomorphic to the standard $S^{3}$, the set 
$\left\{ \widetilde{g}\in \left[ g\right] :\widetilde{Q}=1\right\} $ is
compact in $C^{\infty }$ topology.
\end{theorem}

Indeed let $\widetilde{g}=u^{-4}g$, then $\widetilde{Q}=1$ becomes%
\begin{equation}
Pu=-\frac{1}{2}u^{-7},\quad u\in C^{\infty }\left( M\right) ,u>0.
\label{eq4.17}
\end{equation}%
We can assume $\left( M,g\right) $ is not conformal diffeomorphic to the
standard $S^{3}$, then it follows from Theorem \ref{thm4.3} that $\ker P=0$
and $G_{P}\left( p,q\right) <0$ for all $p,q\in M$. Let $K\left( p,q\right)
=-G_{P}\left( p,q\right) >0$, then (\ref{eq4.17}) becomes%
\begin{equation}
u=\frac{1}{2}T_{K}\left( u^{-7}\right) .  \label{eq4.18}
\end{equation}%
For $0\leq t\leq 1$, we consider a family of integral equations%
\begin{equation}
u=\frac{1}{2}T_{\left( 1-t\right) +tK}\left( u^{-7}\right) .  \label{eq4.19}
\end{equation}%
Elementary apriori estimate for (\ref{eq4.19}) based on the fact $K$ is
bounded and strictly positive together with a degree theory argument gives
us Theorem \ref{thm4.4} (see \cite{HY3}). Note the proof of Theorem \ref%
{thm4.4} is technically simpler than the proof of Theorem \ref{thm3.8}. This
gives a partial solution to Problem \ref{prob4.2}.

To find more solutions to Problem \ref{prob4.2}, we turn our attention to
variational methods. If we write $\widetilde{g}=\rho ^{-4}g$, then the
problem becomes%
\begin{equation}
P\rho =const\cdot \rho ^{-7},\quad \rho \in C^{\infty }\left( M\right) ,\rho
>0.  \label{eq4.20}
\end{equation}%
For $u\in C^{\infty }\left( M\right) $, we denote%
\begin{eqnarray}
E\left( u,v\right) &=&\int_{M}Pu\cdot ud\mu  \label{eq4.21} \\
&=&\int_{M}\left[ \left( \Delta u\right) ^{2}-4Rc\left( \nabla u,\nabla
u\right) +\frac{5}{4}R\left\vert \nabla u\right\vert ^{2}-\frac{1}{2}Qu^{2}%
\right] d\mu  \notag \\
&=&\int_{M}\left[ \left( \Delta u\right) ^{2}-4A\left( \nabla u,\nabla
u\right) +J\left\vert \nabla u\right\vert ^{2}-\frac{1}{2}Qu^{2}\right] d\mu
.  \notag
\end{eqnarray}%
It is clear that $E\left( u\right) $ extends continuously to $u\in
H^{2}\left( M\right) $. Sobolev embedding theorem tells us $H^{2}\left(
M\right) \subset C^{\frac{1}{2}}\left( M\right) $, hence we can set%
\begin{equation}
Y_{4}\left( g\right) =\inf_{u\in H^{2}\left( M\right) ,u>0}E\left( u\right)
\left\Vert u^{-1}\right\Vert _{L^{6}}^{2}=-\frac{1}{2}\sup_{\widetilde{g}\in %
\left[ g\right] }\widetilde{\mu }\left( M\right) ^{\frac{1}{3}}\int_{M}%
\widetilde{Q}d\widetilde{\mu }.  \label{eq4.22}
\end{equation}%
$Y_{4}\left( g\right) $ is a conformal invariant similar to $Y\left(
g\right) $. But unlike $Y\left( g\right) $, it is not clear anymore whether $%
Y_{4}\left( g\right) $ is finite or not.

\begin{problem}
\label{prob4.3}Let $\left( M,g\right) $ be a $3$ dimensional smooth compact
Riemannian manifold, do we have $Y_{4}\left( g\right) >-\infty $? Is $%
Y_{4}\left( g\right) $ always achieved?
\end{problem}

To better understand the problem, following \cite{HY1}, we start with some
basic analysis. Let $u_{i}$ be a minimizing sequence for (\ref{eq4.22}). By
scaling we can assume $\left\Vert u_{i}\right\Vert _{L^{2}}=1$. By Holder
inequality we have%
\begin{equation*}
c=\left\Vert 1\right\Vert _{L^{\frac{3}{2}}}\leq \left\Vert u_{i}\right\Vert
_{L^{2}}\left\Vert u_{i}^{-1}\right\Vert _{L^{6}},
\end{equation*}%
hence%
\begin{equation*}
\left\Vert u_{i}^{-1}\right\Vert _{L^{6}}\geq c>0.
\end{equation*}%
It follows that $E\left( u_{i}\right) \leq c$ and hence $\left\Vert
u_{i}\right\Vert _{H^{2}}\leq c$. After passing to a subsequence we can find 
$u\in H^{2}\left( M\right) $ such that $u_{i}\rightharpoonup u$ weakly in $%
H^{2}\left( M\right) $. It follows that $\left\Vert u\right\Vert _{L^{2}}=1$
and $u\geq 0$.

If $u>0$, then by lower semicontinuity we know $u$ is a minimizer. On the
other hand if $u$ touches zero somewhere, then%
\begin{equation*}
\infty =\left\Vert u^{-1}\right\Vert _{L^{6}}\leq \lim \inf_{i\rightarrow
\infty }\left\Vert u_{i}^{-1}\right\Vert _{L^{6}},
\end{equation*}%
hence%
\begin{equation*}
E\left( u\right) \leq \lim \inf_{i\rightarrow \infty }E\left( u_{i}\right)
\leq 0.
\end{equation*}%
If we can rule out the second case, then $Y_{4}\left( g\right) $ is achieved.

\begin{definition}[\protect\cite{HY1}]
\label{def4.1}Let $\left( M,g\right) $ be a $3$ dimensional smooth compact
Riemannian manifold. If $u\in H^{2}\left( M\right) $ with $u\geq 0$ and $u=0$
somewhere would imply $E\left( u\right) \geq 0$, then we say the metric $g$
(or the associated Paneitz operator) satisfies condition NN$^{+}$. If $u\in
H^{2}\left( M\right) $ is a nonzero function with $u\geq 0$ and $u=0$
somewhere would imply $E\left( u\right) >0$, then we say the metric $g$
satisfies condition P$^{+}$.
\end{definition}

\begin{theorem}[\protect\cite{HY1}]
\label{thm4.5}Let $\left( M,g\right) $ be a $3$ dimensional smooth compact
Riemannian manifold. Then we have%
\begin{equation*}
Y_{4}\left( g\right) \text{ is finite}\Rightarrow g\text{ satisfies NN}^{+}
\end{equation*}%
and%
\begin{equation*}
g\text{ satisfies P}^{+}\Rightarrow Y_{4}\left( g\right) \text{ is achieved
and hence finite.}
\end{equation*}
\end{theorem}

Note condition P$^{+}$ is clearly satisfied when $P>0$. In this case,
Theorem \ref{thm4.5} was proved in \cite{XY2}. Here is an example when we
have positivity of the Paneitz operator.

\begin{lemma}[\protect\cite{HY1}]
\label{lem4.1}If $Y\left( g\right) >0,\sigma _{2}\left( A\right) >0,$ $Q\leq
0$ and not identically zero, then $P>0$.
\end{lemma}

Examples satisfying assumptions in Lemma \ref{lem4.1} can be found in Berger
spheres (see \cite{HY1}). Here we give another criterion for positivity in
the same spirit as \cite[Theorem 1.6]{CHY}.

\begin{lemma}
\label{lem4.2}If $\sigma _{2}\left( A\right) <0$ and $2Jg\geq A$ (note this
implies $J\geq 0$), then $P>0$.
\end{lemma}

\begin{proof}
Let 
\begin{equation*}
\Theta =D^{2}u-\frac{\Delta u}{3}g
\end{equation*}%
be the traceless Hessian and%
\begin{equation*}
\mathring{A}=A-\frac{J}{3}g
\end{equation*}%
be the traceless Schouten tensor. For convenience we use $A\sim B$ to mean $%
\int_{M}Ad\mu =\int_{M}Bd\mu $. First we derive the Bochner identity,%
\begin{eqnarray*}
\left( \Delta u\right) ^{2} &=&u_{ii}u_{jj}\sim -u_{iij}u_{j}=-\left(
u_{iji}-R_{ijik}u_{k}\right) u_{j}\sim u_{ij}u_{ij}+Rc_{jk}u_{j}u_{k} \\
&=&\left\vert D^{2}u\right\vert ^{2}+Rc\left( \nabla u,\nabla u\right)
=\left\vert D^{2}u\right\vert ^{2}+A\left( \nabla u,\nabla u\right)
+J\left\vert \nabla u\right\vert ^{2}.
\end{eqnarray*}%
Hence%
\begin{equation*}
\left( \Delta u\right) ^{2}\sim \left\vert \Theta \right\vert ^{2}+\frac{%
\left( \Delta u\right) ^{2}}{3}+A\left( \nabla u,\nabla u\right)
+J\left\vert \nabla u\right\vert ^{2}.
\end{equation*}%
In another way%
\begin{equation*}
\left( \Delta u\right) ^{2}\sim \frac{3}{2}\left\vert \Theta \right\vert
^{2}+\frac{3}{2}A\left( \nabla u,\nabla u\right) +\frac{3}{2}J\left\vert
\nabla u\right\vert ^{2}.
\end{equation*}%
The next step is to remove the $\Delta J$ term in $Q$ curvature. Note that%
\begin{eqnarray*}
-\Delta J\cdot u^{2} &=&-J_{ii}u^{2}\sim 2J_{i}u\cdot u_{i}=2A_{ijj}u\cdot
u_{i}\sim -2A_{ij}u_{i}u_{j}-2A_{ij}u_{ij}u \\
&=&-2A_{ij}u_{i}u_{j}-2A_{ij}\Theta _{ij}u-\frac{2}{3}Ju\Delta u \\
&=&-2A_{ij}u_{i}u_{j}-2A_{ij}\Theta _{ij}u-\frac{1}{3}J\left( \Delta
u^{2}-2\left\vert \nabla u\right\vert ^{2}\right) \\
&\sim &-2A_{ij}u_{i}u_{j}-2A_{ij}\Theta _{ij}u-\frac{1}{3}\Delta J\cdot
u^{2}+\frac{2}{3}J\left\vert \nabla u\right\vert ^{2}.
\end{eqnarray*}%
Hence%
\begin{eqnarray*}
-\Delta J\cdot u^{2} &\sim &-3A\left( \nabla u,\nabla u\right)
-3A_{ij}\Theta _{ij}u+J\left\vert \nabla u\right\vert ^{2} \\
&=&-3A\left( \nabla u,\nabla u\right) -3\mathring{A}_{ij}\Theta
_{ij}u+J\left\vert \nabla u\right\vert ^{2}.
\end{eqnarray*}%
It follows that%
\begin{eqnarray*}
&&\left( \Delta u\right) ^{2}+J\left\vert \nabla u\right\vert ^{2}-4A\left(
\nabla u,\nabla u\right) -\frac{1}{2}Qu^{2} \\
&\sim &\frac{3}{2}\left\vert \Theta \right\vert ^{2}+\frac{3}{2}\mathring{A}%
_{ij}\Theta _{ij}u+2J\left\vert \nabla u\right\vert ^{2}-A\left( \nabla
u,\nabla u\right) -\frac{3}{4}J^{2}u^{2}+\left\vert A\right\vert ^{2}u^{2} \\
&=&\frac{3}{2}\left\vert \Theta +\frac{1}{2}u\mathring{A}\right\vert
^{2}+2J\left\vert \nabla u\right\vert ^{2}-A\left( \nabla u,\nabla u\right) -%
\frac{5}{8}\left( J^{2}-\left\vert A\right\vert ^{2}\right) u^{2}.
\end{eqnarray*}%
In another word%
\begin{eqnarray*}
E\left( u\right) &=&\frac{3}{2}\int_{M}\left\vert \Theta +\frac{1}{2}u%
\mathring{A}\right\vert ^{2}d\mu +\int_{M}\left[ 2J\left\vert \nabla
u\right\vert ^{2}-A\left( \nabla u,\nabla u\right) \right] d\mu \\
&&-\frac{5}{8}\int_{M}\left( J^{2}-\left\vert A\right\vert ^{2}\right)
u^{2}d\mu .
\end{eqnarray*}%
The positivity follows.
\end{proof}

The assumption in Lemma \ref{lem4.2} is satisfied by $S^{2}\times S^{1}$
with the product metric and some Berger's spheres (see \cite{HY1}).

Conditions P$^{+}$ and NN$^{+}$ are hard to check in general, on the other
hand, they are hard to use too. The closely related conditions P and NN can
be introduced.

\begin{definition}[\protect\cite{HY1}]
\label{def4.2}Let $\left( M,g\right) $ be a $3$ dimensional smooth compact
Riemannian manifold. If $u\in H^{2}\left( M\right) $ with $u=0$ somewhere
would imply $E\left( u\right) \geq 0$, then we say the metric $g$ (or the
associated Paneitz operator) satisfies condition NN. If $u\in H^{2}\left(
M\right) $ is a nonzero function with $u=0$ somewhere would imply $E\left(
u\right) >0$, then we say the metric $g$ satisfies condition P.
\end{definition}

Condition NN can be used to identify the limit function $u$, when $u$
touches zero in the brief discussion after Problem \ref{prob4.3} (see \cite%
{HY1}).

The standard sphere $S^{3}$ does not satisfy condition P$^{+}$. Indeed, let $%
x$ be the coordinate given by the stereographic projection with respect to
north pole $N$, then the Green's function of $P$ at $N$ can be written as%
\begin{equation}
G_{N}=-\frac{1}{4\pi }\frac{1}{\sqrt{\left\vert x\right\vert ^{2}+1}}.
\label{eq4.23}
\end{equation}%
In particular, $E\left( G_{N}\right) =G_{N}\left( N\right) =0$.

\begin{theorem}[\protect\cite{YZ}]
\label{thm4.6}$Y_{4}\left( S^{3},g_{S^{3}}\right) $ is achieved at the
standard metric.
\end{theorem}

Indeed \cite{YZ} shows $Y_{4}\left( S^{3}\right) $ is achieved by the method
of symmetrization. All the critical points are classified by \cite{X}. In 
\cite{H, HY1}, several different approaches are given. The main ingredient
is the following observation:

\begin{lemma}[\protect\cite{HY1}]
\label{lem4.3}Let $N\in S^{3}$ be the north pole, $u\in H^{2}\left(
S^{3}\right) $ such that $u\left( N\right) =0$. Denote $x$ as the coordinate
given by the stereographic projection with respect to $N$ and%
\begin{equation*}
\tau =\sqrt{\frac{\left\vert x\right\vert ^{2}+1}{2}}.
\end{equation*}%
Then we know $\Delta \left( \tau u\right) \in L^{2}\left( \mathbb{R}%
^{3}\right) $ and 
\begin{equation}
E\left( u\right) =\int_{\mathbb{R}^{3}}\left\vert \Delta \left( \tau
u\right) \right\vert ^{2}dx,  \label{eq4.24}
\end{equation}%
here $\Delta $ is the Euclidean Laplacian.
\end{lemma}

In particular $S^{3}$ satisfies NN. The only functions touching $0$ and
having nonpositive energy are constant multiples of Green's functions.

To help understanding the condition NN, in \cite{HY2}, new quantities $\nu
\left( M,g,p\right) $ and $\nu \left( M,g\right) $ are introduced. Let $%
\left( M,g\right) $ be a $3$ dimensional smooth compact Riemannian manifold,
for any $p\in M$, define%
\begin{equation}
\nu \left( M,g,p\right) =\inf \left\{ \frac{E\left( u\right) }{%
\int_{M}u^{2}d\mu }:u\in H^{2}\left( M\right) \backslash \left\{ 0\right\}
,u\left( p\right) =0\right\} .  \label{eq4.25}
\end{equation}%
When no confusion could arise we denote it as $\nu \left( g,p\right) $ or $%
\nu _{p}$. We also define%
\begin{eqnarray}
&&\nu \left( M,g\right)  \label{eq4.26} \\
&=&\inf_{p\in M}\nu \left( M,g,p\right)  \notag \\
&=&\inf \left\{ \frac{E\left( u\right) }{\int_{M}u^{2}d\mu }:u\in
H^{2}\left( M\right) \backslash \left\{ 0\right\} ,u\left( p\right) =0\text{
for some }p\right\} .  \notag
\end{eqnarray}%
The importance of $\nu \left( M,g\right) $ lies in that $g$ satisfies
condition P if and only if $\nu \left( g\right) >0$ and it satisfies
condition NN if and only if $\nu \left( g\right) \geq 0$. It follows from
Lemma \ref{lem4.3} that $\nu \left( S^{3},g_{S^{3}}\right) =0$. A closely
related fact is that the Green's function of Paneitz operator on $S^{3}$
vanishes at the pole. In \cite{HY2}, first and second variation of $%
G_{P}\left( N,N\right) $ and $\nu \left( S^{3},g,N\right) $ are calculated.

\begin{theorem}[\protect\cite{HY2}]
\label{thm4.7}Let $g$ be the standard metric on $S^{3}$and $h$ be a smooth
symmetric $\left( 0,2\right) $ tensor. Denote $x=\pi _{N}$, the
stereographic projection with respect to $N$ and%
\begin{equation*}
\tau =\sqrt{\frac{\left\vert x\right\vert ^{2}+1}{2}}.
\end{equation*}
Let $G_{g+th}$ be the Green's function of the Paneitz operator $P_{g+th}$,
then%
\begin{equation}
\left. \partial _{t}\right\vert _{t=0}G_{g+th}\left( N,N\right) =0
\label{eq4.27}
\end{equation}%
and%
\begin{eqnarray}
&&\left. \partial _{t}^{2}\right\vert _{t=0}G_{g+th}\left( N,N\right)
\label{eq4.28} \\
&=&-\frac{1}{64\pi ^{2}}\int_{\mathbb{R}^{3}}\left( \sum_{ij}\left( \theta
_{ikjk}+\theta _{jkik}-\left( tr\theta \right) _{ij}-\Delta \theta
_{ij}\right) ^{2}-\frac{3}{2}\left( \theta _{ijij}-\Delta tr\theta \right)
^{2}\right) dx.  \notag
\end{eqnarray}%
Here $\theta =\tau ^{4}h$ and the derivatives $\theta _{ikjk}$ etc are
partial derivatives in $\mathbb{R}^{3}$.

In particular, 
\begin{equation}
\left. \partial _{t}^{2}\right\vert _{t=0}G_{g+th}\left( N,N\right) \leq 0
\label{eq4.29}
\end{equation}%
Moreover, $\left. \partial _{t}^{2}\right\vert _{t=0}G_{g+th}\left(
N,N\right) =0$ if and only if $h=L_{X}g+f\cdot g$ for some smooth vector
fields $X$ and smooth function $f$ on $S^{3}$.

For $\nu \left( g+th,N\right) $ we have%
\begin{equation}
\left. \partial _{t}\right\vert _{t=0}\nu \left( g+th,N\right) =0
\label{eq4.30}
\end{equation}%
and%
\begin{equation}
\left. \partial _{t}^{2}\right\vert _{t=0}\nu \left( g+th,N\right)
=-16\left. \partial _{t}^{2}\right\vert _{t=0}G_{g+th}\left( N,N\right) .
\label{eq4.31}
\end{equation}
\end{theorem}

In \cite{HY2}, a close relation between condition NN and the second
eigenvalue of Paneitz operator is given.

\begin{theorem}[\protect\cite{HY2}]
\label{thm4.8}Let $\left( M,g\right) $ be a $3$ dimensional smooth compact
Riemannian manifold with $Y\left( g\right) >0$ and $r_{\sigma }\left(
T_{\Gamma _{1}}\right) <1$, then the following statements are equivalent:

\begin{enumerate}
\item $Y_{4}\left( g\right) >-\infty $.

\item $\lambda _{2}\left( P\right) >0$.

\item $\nu \left( g\right) \geq 0$ i.e. $\left( M,g\right) $ satisfies
condition NN.
\end{enumerate}
\end{theorem}

For condition P, there is a similar statement.

\begin{corollary}[\protect\cite{HY2}]
\label{cor4.1}Let $\left( M,g\right) $ be a $3$ dimensional smooth compact
Riemannian manifold with $Y\left( g\right) >0$ and $r_{\sigma }\left(
T_{\Gamma _{1}}\right) <1$. If $\left( M,g\right) $ is not conformal
diffeomorphic to the standard $S^{3}$, then the following statements are
equivalent:

\begin{enumerate}
\item $Y_{4}\left( g\right) >-\infty $.

\item $\lambda _{2}\left( P\right) >0$.

\item $\nu \left( g\right) >0$ i.e. $\left( M,g\right) $ satisfies condition
P.
\end{enumerate}
\end{corollary}

These statements make the finiteness of $Y_{4}\left( g\right) $ and
condition NN more meaningful.

\begin{problem}
\label{prob4.4}Let $\left( M,g\right) $ be a $3$ dimensional smooth compact
Riemannian manifold, does $g$ always satisfy condition NN? Does metric with
positive Yamabe invariant always satisfy condition NN?
\end{problem}

This seems to be a difficult question. We only have a partial answer.

\begin{theorem}[\protect\cite{HY3, HY4}]
\label{thm4.9}Assume $M$ is a smooth compact $3$ dimensional manifold, denote%
\begin{equation}
\mathcal{M}=\left\{ g:Y\left( g\right) >0\text{, }r_{\sigma }\left(
T_{\Gamma _{1}}\right) <1\right\} ,  \label{eq4.32}
\end{equation}%
endowed with $C^{\infty }$ topology. Let $\mathcal{N}$ be a path connected
component of $\mathcal{M}$. If there is a metric in $\mathcal{N}$ satisfying
condition NN, then every metric in $\mathcal{N}$ satisfies condition NN.
Hence as long as the metric is not conformal equivalent to the standard $%
S^{3}$, it satisfies condition P.
\end{theorem}

Here we describe an application of above discussions. Let $M$ be a $3$
dimensional smooth compact manifold, $\beta \in \mathbb{R}$, we define a
functional%
\begin{equation}
F_{\beta }\left( g\right) =\int_{M}Qd\mu +\beta \int_{M}J^{2}d\mu .
\label{eq4.33}
\end{equation}%
Calculation shows that the critical metric of $F_{\beta }$ restricted to a
fixed conformal class with unit volume constraint is given by%
\begin{equation}
Q+2\beta \Delta J+\beta J^{2}=const.  \label{eq4.34}
\end{equation}

\begin{proposition}
\label{prop4.1}Assume $\left( M,g\right) $ is a $3$ dimensional smooth
compact Riemannian manifold satisfying condition P$^{+}$, $\beta \leq 0$,
then%
\begin{equation}
\sup_{\widetilde{g}\in \left[ g\right] }\widetilde{\mu }\left( M\right) ^{%
\frac{1}{3}}F_{\beta }\left( \widetilde{g}\right)  \label{eq4.35}
\end{equation}%
is achieved.
\end{proposition}

\begin{proof}
For any positive smooth function $u$,%
\begin{equation}
F_{\beta }\left( u^{-4}g\right) =-2\int_{M}Pu\cdot ud\mu +\beta
\int_{M}u^{4} \left[ -2\Delta \left( u^{-1}\right) +Ju^{-1}\right] ^{2}d\mu .
\label{eq4.36}
\end{equation}%
Define

\begin{equation}
\Phi _{\beta }\left( u\right) =\int_{M}Pu\cdot ud\mu -\frac{\beta }{2}%
\int_{M}u^{4}\left[ -2\Delta \left( u^{-1}\right) +Ju^{-1}\right] ^{2}d\mu ,
\label{eq4.37}
\end{equation}%
then%
\begin{equation}
\sup_{\widetilde{g}\in \left[ g\right] }\widetilde{\mu }\left( M\right) ^{%
\frac{1}{3}}F_{\beta }\left( \widetilde{g}\right) =-2\inf_{\substack{ u\in
C^{\infty }\left( M\right)  \\ u>0}}\left\Vert u^{-1}\right\Vert
_{L^{6}}^{2}\Phi _{\beta }\left( u\right) .  \label{eq4.38}
\end{equation}

Let%
\begin{equation}
m=\inf_{\substack{ u\in H^{2}\left( M\right)  \\ u>0}}\left\Vert
u^{-1}\right\Vert _{L^{6}}^{2}\Phi _{\beta }\left( u\right) .  \label{eq4.39}
\end{equation}%
We claim $m$ is achieved. Indeed%
\begin{eqnarray}
\Phi _{\beta }\left( u\right) &=&\int_{M}\left[ \left( \Delta u\right)
^{2}-4A\left( \nabla u,\nabla u\right) +J\left\vert \nabla u\right\vert ^{2}-%
\frac{1}{2}Qu^{2}\right] d\mu  \label{eq4.40} \\
&&-2\beta \int_{M}\left( \Delta u-2u^{-1}\left\vert \nabla u\right\vert ^{2}+%
\frac{J}{2}u\right) ^{2}d\mu .  \notag
\end{eqnarray}%
Assume $u_{i}\in H^{2}\left( M\right) $, $u_{i}>0$ is a minimizing sequence,
by scaling we can assume $\max_{M}u_{i}=1$. Then%
\begin{equation*}
\left\Vert u_{i}^{-1}\right\Vert _{L^{6}}^{2}\Phi _{\beta }\left(
u_{i}\right) \rightarrow m.
\end{equation*}%
It follows from $\beta \leq 0$ that%
\begin{equation*}
\left\Vert u_{i}^{-1}\right\Vert _{L^{6}}^{2}E\left( u_{i}\right) \leq c.
\end{equation*}%
Hence $E\left( u_{i}\right) \leq c$. Together with the fact $0<u_{i}\leq 1$
we get $\left\Vert u_{i}\right\Vert _{H^{2}\left( M\right) }\leq c$. After
passing to a subsequence we can assume $u_{i}\rightharpoonup u$ weakly in $%
H^{2}\left( M\right) $. Then $u_{i}\rightarrow u$ uniformly. It follows that 
$\max_{M}u=1$ and $u\geq 0$. We claim $u$ can not touch $0$. Indeed if $u$
touches zero somewhere, then since $u\in H^{2}\left( M\right) $ we see $%
\int_{M}u^{-6}d\mu =\infty $. It follows from Fatou's lemma that%
\begin{equation*}
\lim \inf_{i\rightarrow \infty }\int_{M}u_{i}^{-6}d\mu \geq
\int_{M}u^{-6}d\mu =\infty .
\end{equation*}%
Hence%
\begin{equation*}
\lim \sup_{i\rightarrow \infty }E\left( u_{i}\right) \leq 0.
\end{equation*}%
It follows that $E\left( u\right) \leq 0$, this contradicts with condition P$%
^{+}$. The fact $u>0$ follows. To continue, we observe that $%
u_{i}\rightarrow u$ in $W^{1,p}\left( M\right) $ for $p<6$. Hence%
\begin{equation*}
\Phi _{\beta }\left( u\right) \leq \lim \inf_{i\rightarrow \infty }\Phi
_{\beta }\left( u_{i}\right) .
\end{equation*}%
It follows that%
\begin{equation*}
\left\Vert u^{-1}\right\Vert _{L^{6}}^{2}\Phi _{\beta }\left( u\right) \leq
\lim \inf_{i\rightarrow \infty }\left\Vert u_{i}^{-1}\right\Vert
_{L^{6}}^{2}\Phi _{\beta }\left( u_{i}\right) =m.
\end{equation*}%
$u$ is a minimizer. Calculation shows for any $\varphi \in C^{\infty }\left(
M\right) $,%
\begin{eqnarray*}
&&\int_{M}uP\varphi d\mu -2\beta \int_{M}\left( \Delta u-2u^{-1}\left\vert
\nabla u\right\vert ^{2}+\frac{J}{2}u\right) \\
&&\cdot \left( \Delta \varphi -4u^{-1}\left\langle \nabla u,\nabla \varphi
\right\rangle +2u^{-2}\left\vert \nabla u\right\vert ^{2}\varphi +\frac{J}{2}%
\varphi \right) d\mu \\
&=&const\int_{M}u^{-7}\varphi d\mu .
\end{eqnarray*}%
Standard bootstrap method shows $u\in C^{\infty }\left( M\right) $.
Proposition \ref{prop4.1} follows.
\end{proof}

\begin{corollary}
\label{cor4.2}Let $\left( M,g\right) $ be a $3$ dimensional smooth compact
Riemannian manifold satisfying condition P$^{+}$, $Y\left( g\right) >0$ and%
\begin{equation}
\int_{M}Qd\mu -\frac{1}{6}\int_{M}J^{2}d\mu \geq 0,  \label{eq4.41}
\end{equation}%
then the universal cover of $M$ is diffeomorphic to $S^{3}$.
\end{corollary}

\begin{remark}
\label{rmk4.2}(\ref{eq4.41}) is the same as%
\begin{equation}
\int_{M}\left\vert E\right\vert ^{2}d\mu \leq \frac{1}{48}\int_{M}R^{2}d\mu ,
\label{eq4.42}
\end{equation}%
here $E=Rc-\frac{R}{3}g$ is the traceless Ricci tensor.
\end{remark}

\begin{proof}
It follows from Proposition \ref{prop4.1} that%
\begin{equation}
\kappa =\sup_{\widetilde{g}\in \left[ g\right] }\widetilde{\mu }\left(
M\right) ^{\frac{1}{3}}\left( \int_{M}\widetilde{Q}d\widetilde{\mu }-\frac{1%
}{6}\int_{M}\widetilde{J}^{2}d\widetilde{\mu }\right)  \label{eq4.43}
\end{equation}%
is achieved. By (\ref{eq4.41}) we know $\kappa \geq 0$. Without losing of
generality we can assume the background metric $g$ is a maximizer and it has
volume $1$. Then%
\begin{equation}
Q-\frac{1}{3}\Delta J-\frac{1}{6}J^{2}=const.  \label{eq4.44}
\end{equation}%
Integrating both sides we get%
\begin{equation}
Q-\frac{1}{3}\Delta J-\frac{1}{6}J^{2}=\kappa \geq 0.  \label{eq4.45}
\end{equation}%
In another way it is%
\begin{equation}
-\frac{4}{3}\Delta J-2\left\vert \mathring{A}\right\vert ^{2}+\frac{2}{3}%
J^{2}=\kappa .  \label{eq4.46}
\end{equation}%
Here $\mathring{A}=A-\frac{J}{3}g$ is the traceless Schouten tensor. Since
the conformal Laplacian is given by $L=-8\Delta +4J$,%
\begin{equation}
LJ=6\kappa +12\left\vert \mathring{A}\right\vert ^{2}\geq 0.  \label{eq4.47}
\end{equation}%
Since $Y\left( g\right) >0$, we see the Green's function of $L$ must be
positive, hence either $J>0$ or $J\equiv 0$. The latter case contradicts
with the fact $Y\left( g\right) >0$. So the scalar curvature must be
strictly positive. Finally%
\begin{equation}
\int_{M}\sigma _{2}\left( A\right) d\mu =\frac{1}{4}\int_{M}\left( Q+\frac{1%
}{2}J^{2}\right) d\mu =\frac{\kappa }{4}+\frac{1}{6}\int_{M}J^{2}d\mu >0.
\label{eq4.48}
\end{equation}%
It follows from a result in \cite{CD, GLW} that the universal cover of $M$
is diffeomorphic to $S^{3}$.
\end{proof}

The above example also shows the interest in the following question:

\begin{problem}
\label{prob4.5}Let $\left( M,g\right) $ be a $3$ dimensional smooth compact
Riemannian manifold, can we find a conformal invariant condition which is
equivalent to the existence of a conformal metric with positive scalar and $%
Q $ curvature?
\end{problem}

We remark that by modifying the technique in \cite{GM}, it is shown in \cite%
{HY3} that on $3$ dimensional smooth compact Riemannian manifolds $\left(
M,g\right) $ with $R>0$ and $Q>0$, if $\widetilde{g}\in \left[ g\right] $
satisfies $\widetilde{Q}>0$, then $\widetilde{R}>0$ too.

\subsection*{Acknowledgement}

The research of Yang is supported by NSF grant 1104536. We would like to
thank Jie Qing and Xingwang Xu for their invitations to talk about the
results in this set of lecture notes.


\begin{thebibliography}{CGY1}
\bibitem[Ad]{Ad} D. Adams. \textit{A sharp inequality of J. Moser for higher
order derivatives. }Ann. of Math. (2) \textbf{128} (1988), no. 2, 385--398.

\bibitem[Au1]{Au1} T. Aubin. \textit{Fonction de Green et valeurs propres du
laplacien.} J. Math. Pures Appl. (9) \textbf{53} (1974), 347--371.

\bibitem[Au2]{Au2} T. Aubin. Nonlinear analysis on manifolds, Monge-Ampere
equations. Springer Verlag, New York, 1982.

\bibitem[B]{B} T. Branson. \textit{Differential operators canonically
associated to a conformal structure}. Math. Scand. \textbf{57} (1985), no.
2, 293--345.

\bibitem[BCY]{BCY} T. Branson, S. Y. Chang and P. C. Yang. \textit{Estimates
and extremals for zeta function determinants on four-manifolds.} Comm. Math.
Phys. \textbf{149} (1992), no. 2, 241--262.

\bibitem[CD]{CD} G. Catino and Z. Djadli. \textit{Conformal deformations of
integral pinched 3-manifolds.} Adv. Math. \textbf{223} (2010), no. 2,
393--404.

\bibitem[CGY1]{CGY1} S. Y. Chang, M. J. Gursky and P. C. Yang. \textit{An
equation of Monge-Ampere type in conformal geometry, and four-manifolds of
positive Ricci curvature}. Ann. of Math. (2) \textbf{155} (2002), no. 3,
709--787.

\bibitem[CGY2]{CGY2} S. Y. Chang, M. J. Gursky and P. C. Yang. \textit{A
conformally invariant sphere theorem in four dimensions}. Publ. Math. Inst.
Hautes Etudes Sci. No. \textbf{98} (2003), 105--143.

\bibitem[CHY]{CHY} S. Y. Chang, F. B. Hang and P. C.Yang. \textit{On a class
of locally conformally flat manifolds}. Int. Math. Res. Not. \textbf{2004},
no. 4, 185--209.

\bibitem[CY]{CY} S. Y. Chang and P. C. Yang. \textit{Extremal metrics of
zeta function determinants on 4-manifolds.} Ann. of Math. (2) \textbf{142}
(1995), no. 1, 171--212.

\bibitem[CLO]{CLO} W. X. Chen, C. M. Li and B. Ou. \textit{Classification of
solutions for an integral equation.} Comm Pure Appl Math \textbf{59} (2006),
no. 3, 330--343.

\bibitem[DM]{DM} Z. Djadli and A. Malchiodi. \textit{Existence of conformal
metrics with constant Q curvature}. Ann. of Math. (2) \textbf{168} (2008),
no. 3, 813--858.

\bibitem[F]{F} L. Fontana. \textit{Sharp borderline Sobolev inequalities on
compact Riemannian manifolds.} Comment. Math. Helv. \textbf{68} (1993), no.
3, 415--454.

\bibitem[GLW]{GLW} Y. X. Ge, C. S. Lin and G. F. Wang. \textit{On the }$%
\sigma _{2}$\textit{-scalar curvature.} J. Differential Geom. \textbf{84}
(2010), no. 1, 45--86.

\bibitem[G1]{G1} M. J. Gursky. \textit{The Weyl functional, de Rham
cohomology, and Kahler-Einstein metrics.} Ann. of Math. (2) \textbf{148}
(1998), no. 1, 315--337.

\bibitem[G2]{G2} M. J. Gursky. \textit{The principal eigenvalue of a
conformally invariant differential operator.} Comm. Math. Phys. \textbf{207}
(1999), no. 1, 131--143.

\bibitem[GHL]{GHL} M. J. Gursky, F. B. Hang and Y. J. Lin. \textit{%
Riemannian manifolds with positive Yamabe invariant and Paneitz operator.}
International Mathematics Research Notices 2015; doi:10.1093/imrn/rnv176.

\bibitem[GM]{GM} M. J. Gursky and A. Malchiodi. \textit{A strong maximum
principle for the Paneitz operator and a nonlocal flow for the }$Q$\textit{\
curvature}. J. Eur. Math. Soc. \textbf{17 }(2015), 2137--2173.

\bibitem[H]{H} F. B. Hang. \textit{On the higher order conformal covariant
operators on the sphere.} Commun Contemp Math. \textbf{9} (2007), no. 3,
279--299.

\bibitem[HWY1]{HWY1} F. B. Hang, X. D. Wang and X. D. Yan. \textit{Sharp
integral inequalities for harmonic functions.} Comm. Pure Appl. Math. 
\textbf{61} (2008), no. 1, 54--95.

\bibitem[HWY2]{HWY2} F. B. Hang, X. D. Wang and X. D. Yan. \textit{An
integral equation in conformal geometry.} Ann. Inst. H. Poincare Anal. Non
Lineaire \textbf{26} (2009), no. 1, 1--21.

\bibitem[HY1]{HY1} F. B. Hang and P. C. Yang. \textit{The Sobolev inequality
for Paneitz operator on three manifolds.} Calculus of Variations and PDE. 
\textbf{21} (2004), 57--83.

\bibitem[HY2]{HY2} F. B. Hang and P. C. Yang. Paneitz operator for metrics
near $S^{3}$. Preprint (2014), arXiv:1504.02032.

\bibitem[HY3]{HY3} F. B. Hang and P. C. Yang. $Q$\textit{\ curvature on a
class of }$3$\textit{\ manifolds.} Comm Pure Appl Math, to appear, DOI:
10.1002/cpa.21559.

\bibitem[HY4]{HY4} F. B. Hang and P. C. Yang. \textit{Sign of Green's
function of Paneitz operators and the }$Q$\textit{\ curvature}.
International Mathematics Research Notices 2014; doi: 10.1093/imrn/rnu247.

\bibitem[HY5]{HY5} F. B. Hang and P. C. Yang. $Q$ curvature on a class of
manifolds with dimension at least $5$. Preprint (2014), arXiv:1411.3926.

\bibitem[HeR1]{HeR1} E. Hebey and F. Robert.\textit{\ Compactness and global
estimates for the geometric Paneitz equation in high dimensions}. Electron
Res Ann Amer Math Soc. \textbf{10} (2004), 135--141.

\bibitem[HeR2]{HeR2} E. Hebey and F. Robert. \textit{Asymptotic analysis for
fourth order Paneitz equations with critical growth}. Adv Calc Var. \textbf{4%
} (2011), no. 3, 229--275.

\bibitem[HuR]{HuR} E. Humbert and S. Raulot. \textit{Positive mass theorem
for the Paneitz-Branson operator.} Calculus of Variations and PDE. \textbf{36%
} (2009), 525--531.

\bibitem[KW]{KW} J. L. Kazdan and F. W. Warner. \textit{Scalar curvature and
conformal deformation of Riemannian structure. }J. Diff. Geom. \textbf{10 }%
(1975), 113--134.

\bibitem[LP]{LP} J. M. Lee and T. H. Parker. \textit{The Yamabe problem}.
Bull AMS. \textbf{17 }(1987), no. 1, 37--91.

\bibitem[L]{L} G. Li. A compactness theorem on Branson's Q-curvature
equation. Preprint (2015), arXiv:1505.07692.

\bibitem[LX]{LX} Y. Y. Li and J. G. Xiong. Compactness of conformal metrics
with constant Q-curvature. I. Preprint (2015), arXiv:1506.00739.

\bibitem[M]{M} J. Moser. \textit{A sharp form of an inequality by N.
Trudinger.} Indiana Univ. Math. J. \textbf{20} (1970/71), 1077--1092.

\bibitem[O]{O} R. O'Neil. \textit{Convolution operators and L(p,q) spaces.}
Duke Math. J. \textbf{30} (1963), 129--142.

\bibitem[P]{P} S. Paneitz. A quartic conformally covariant differential
operator for arbitrary pseudo-Riemannian manifolds. Preprint (1983),
arXiv:0803.4331.

\bibitem[QR1]{QR1} J. Qing and D. Raske. \textit{Compactness for conformal
metrics with constant Q curvature on locally conformally flat manifolds.}
Calc. Var. Partial Differential Equations \textbf{26} (2006), no. 3,
343--356.

\bibitem[QR2]{QR2} J. Qing and D. Raske. \textit{On positive solutions to
semilinear conformally invariant equations on locally conformally flat
manifolds}. Int Math Res Not. Art. id 94172 (2006).

\bibitem[R]{R} F. Robert. Fourth order equations with critical growth in
Riemannian geometry. Unpublished notes. Available at
http://www.iecn.u-nancy.fr/\symbol{126}frobert/LectRobertFourth.pdf.

\bibitem[S]{S} R. Schoen. \textit{Conformal deformation of a Riemannian
metric to constant scalar curvature.} J. Differential Geom. \textbf{20}
(1984), no. 2, 479--495.

\bibitem[SY]{SY} R. Schoen and S. T. Yau. \textit{Conformally flat
manifolds, Kleinian groups and scalar curvature.} Inventiones Mathematicae 
\textbf{92} (1988) no. 1, 47--71.

\bibitem[T]{T} N. Trudinger. \textit{Remarks concerning the conformal
deformation of Riemannian structures on compact manifolds.} Ann. Scuola
Norm. Sup. Pisa (3) \textbf{22} (1968), 265--274.

\bibitem[WZ]{WZ} J. C. Wei and C. Y. Zhao. \textit{Non-compactness of the
prescribed Q-curvature problem in large dimensions}. Calc. Var. Partial
Differential Equations \textbf{46} (2013), no. 1-2, 123--164.

\bibitem[X]{X} X. W. Xu. \textit{Exact solutions of nonlinear conformally
invariant integral equations in }$\mathbb{R}^{3}$. Adv. Math. \textbf{194}
(2005), no. 2, 485--503.

\bibitem[XY1]{XY1} X. W. Xu and P. C. Yang. \textit{Positivity of Paneitz
operators}. Discrete Contin. Dynam. Systems \textbf{7} (2001), no. 2,
329--342.

\bibitem[XY2]{XY2} X. W. Xu and P. C. Yang. \textit{On a fourth order
equation in 3-D}. ESAIM: Control, optimization and Calculus of Variations 
\textbf{8 }(2002) 1029--1042.

\bibitem[Y]{Y} H. Yamabe. \textit{On a deformation of Riemannian structures
on compact manifolds.} Osaka Math. J. \textbf{12} (1960), 21--37.

\bibitem[YZ]{YZ} P. C. Yang and M. J. Zhu. \textit{Sharp inequality for
Paneitz operator on }$S^{3}$. ESAIM: Control, Optim. and Calc. of Var. 
\textbf{10} (2004) 211--223.
\end{thebibliography}
\end{document}